%% file: Time-Inconsistent-09-04-20-tom.tex
\documentclass[11pt,twoside]{article}
\input{headingsIgor2020}

\title{ \vspace{-3em} 
   Time-inconsistent Markovian control problems under model uncertainty with application to the mean-variance portfolio selection
}

\makeatletter
\def\and{%
  \end{tabular}%
  \begin{tabular}[t]{c}}%
\def\@fnsymbol#1{\ensuremath{\ifcase#1\or a\or b\or c\or
   d\or e\or f\or g\or h\or i\else\@ctrerr\fi}}
\makeatother

\author{
        Tomasz R. Bielecki\,\thanks{Department of Applied Mathematics, Illinois Institute of Technology
       \newline \hspace*{1.45em}  10 W 32nd Str, Building RE, Room 220, Chicago, IL 60616, USA
       \newline \hspace*{1.45em}  Emails: \url{tbielecki@iit.edu} (T. R. Bielecki), and \url{cialenco@iit.edu} (I. Cialenco)
       \newline \hspace*{1.45em}  URLs: \url{http://math.iit.edu/\~bielecki}  and \url{http://cialenco.com}
        \vspace{0.5em}} ,
\and
        Tao Chen\,\thanks{Department of Mathematics, University of Michigan, Ann Arbor
        \newline \hspace*{1.45em}  530 Church Street, East Hall, Room 2859, Ann Arbor, MI 48109, USA
        \newline \hspace*{1.45em} Email: \url{chenta@umich.edu}, URL: \url{http://taochen.im}
        \vspace{0.5em}},
\and
         Igor Cialenco,\,\footnotemark[1]        
        }

\date{ {\small 
First Circulated: February 6, 2020\\
This Version: September 8, 2020
}}

\begin{document}

\maketitle

{\footnotesize
\begin{tabular}{l@{} p{350pt}}
  \hline \\[-.2em]
  \textsc{Abstract}: \ &  In this paper we study  a class of  time-inconsistent terminal Markovian control problems in discrete time subject to model uncertainty. We combine the concept of the sub-game perfect strategies with the adaptive robust stochastic control method to tackle the theoretical aspects of the considered stochastic control problem. Consequently, as an important application of the  theoretical results and by applying a machine learning algorithm we solve numerically the mean-variance portfolio selection problem under the model uncertainty.
  \\[1em]
\textsc{Keywords:} \ & adaptive robust control, model uncertainty, stochastic control, adaptive robust dynamic programming, recursive confidence regions, time-inconsistent Markovian control problem, optimal portfolio allocation, mean-variance portfolio selection, terminal criteria, {machine learning, Gaussian surrogate processes, regression Monte Carlo.}
 \\[.5em]
\textsc{MSC2010:} \ & 49L20, 60J05, 60J20, 91A10, 91G10, 91G80, 62F25, 93C40, 93E35  \\[1em]
  \hline
\end{tabular}
}


\section{Introduction}

The main goal of this study is to develop a methodology to solve efficiently some \textbf{time-inconsistent} Markovian control problems subject to \textbf{model uncertainty} in a discrete time setup. The proposed approach hinges on the following main building concepts: first, incorporating model uncertainty through the \textbf{adaptive robust} paradigm introduced in \cite{BCCCJ2019}; second, dealing with time-inconsistency of the stochastic control problem at hand by exploiting  the concept of \textbf{sub-game perfect} strategies as studied in \cite{BjoerkMurgoci2014}; third, developing efficient numerical solutions for the obtained Bellman equations by adopting the \textbf{machine learning} techniques proposed in \cite{ChenLudkovski2019}.

There exists a significant body of work on incorporating model uncertainty (or model misspecification) in stochastic control problems, and among some of the well-known and prominent methods we would mention the robust control approach \cite{GilboaSchmeidler1989,HansenSargent2006,HansenBookBook2008}, adaptive control \cite{ChenGuo1991-Book,DuncanStettner2001,DuncanStettner2006,KumarVaraiya2015Book}, and Bayesian adaptive control \cite{KumarVaraiya2015Book}. We refer the reader to \cite{BCCCJ2019} for a relatively  comprehensive literature review on this subject and their connection to the adaptive robust methodology used in this paper and originally introduced in~\cite{BCCCJ2019}. The adaptive robust methodology of \cite{BCCCJ2019} is an approach that solves (time-consistent) Markovian control problems in discrete time subject to model uncertainty. The core of this methodology was to combine a recursive learning mechanism about the unknown model, with the Markovian dynamics of that model and with the time-consistent nature of the control problem studied therein, which allowed to derive an adequate system of recursive dynamic programming equations, which where dubbed the adaptive robust Bellman equations that gave a solution to the original control problem.

In all the above mentioned methods, inherently the stochastic problems are (strongly) time-consistent in the sense that the dynamic programming principle holds true. For an overview of the time consistency in decision making, cf. \cite{BCP2014a,BCP2014}. While lack of (strong) time consistency in decision making is not necessarily an unacceptable feature, from stochastic control point of view it may lead to undesirable properties that, in particular, may lack adequate numerically tractable solutions.  A good body of literature have been dedicated to time-inconsistent stochastic control problems have been emerged in the recent years, primarily for continuous time setup. We refer the reader to \cite[Chpater~2]{He2018} and \cite{ShiCui2017} for a comprehensive literature review of time-inconsistent stochastic control problems in discrete time.
Broadly speaking,  there are three avenues that researchers followed in dealing with time-inconsistent Markovian control problems (primarily in discrete time) when the underlying model is fully known:
\begin{itemize}
\item[1)]  The \textit{pre-commitment approach} emphasizes the global optimality, namely the controller optimizes the expected objective functional at the initial time point and sticks to the resulting strategy through the whole time period. In general, such strategy will be time-inconsistent if it is not revised by the controller afterwards.
In the context of optimal portfolio selection, the authors of \cite{LiNg2000} and \cite{LiZhou2000} introduced an embedding technique to obtain a pre-committed solution of the dynamic mean-variance optimization problem.
	
\item[2)] The  \textit{sub-game perfect} approach, sometimes also called consistent planning approach,  is based on ideas rooted in game theory, where the time-consistent strategy is derived by assuming that the investor is playing an optimization game with future-self. This approach, originated in \cite{Strotz1955,Goldman1980} and systemically studied in \cite{EkelandLazrak2006,EkelandPirvu2008,EkelandLazrak2010,HuJinZhou2012,HernandezPossamai2020}   (in continuous time) leads to a specific notion of optimality (sub-game perfection), which can be characterized in terms of respective dynamic programming equations. In this regard, \cite{BasakChabakauri2010,BjoerkMurgoci2014} were the first to apply and extend this approach to the mean-variance problem. Some more different examples are investigated in \cite{BjoerkEtAl2014,Bannister2016}.

\item[3)] By \textit{modifying the criteria} as time evolves such that the dynamic programming principle holds that has been studied, in various forms in \cite{BouchardEtAl2010,CuiLiWangEtAl2012,KarnamEtAl2017,FeinsteinRudloff2019,KovacovaRudloff2019}.
\end{itemize}
It should be mentioned that only a selected number of publications on time-inconsistency of stochastic control problems in continuous time setup were mentioned here. Generally speaking, there is no one right method in addressing the time-inconsistency, and each of these three approaches has its advantages and drawbacks.

In this work, we are focusing on the sub-game perfect approach that is appropriately formulated for the Markovian control problem with model uncertainty. As an important application of the proposed general theory, we consider the mean-variance portfolio selection problem under model uncertainty. Besides being an important contribution of our paper, arguably, the classical mean-variance portfolio optimization methodology is one of the most  popular portfolio selection methodology among managers of financial portfolios. Notably, majority of the above cited literature is devoted to the mean-variance optimization problem.

It is well-documented that solving numerically stochastic control problems subject to model uncertainty is a challenging task, and classical numerical methods can not be successfully applied even to the simplest problems. In \cite{ChenLudkovski2019}  the authors introduced a method, rooted in the machine learning methodology, to deal with such problems in the context of an adaptive robust, time-consistent, stochastic control problem. In the present work, we apply a similar computational approach for solving the aforementioned mean-variance problem.

The paper is organized as follows. In Section~\ref{sec:robust} we formulate the time-inconsistent Markovian control problem subject to model uncertainty, as well as, the corresponding time-inconsistent  adaptive robust control problem (see Section~\ref{sec:formulation}). The main theoretical developments of this work are presented in Section~\ref{sec:ARSG}. Specifically, in this section we propose and analyze the time consistent sub-game perfect approach to deal  with the adaptive robust control problem of Section~\ref{sec:formulation}. We derive the Bellman equations for the sub-game perfect strategies; see Theorem~\ref{thm:1}.  In Section~\ref{sec:existence} we study the existence of sub-game perfect strategies. An illustrative example of our theoretical results that is rooted in the classical Markowitz's mean-variance portfolio theory, is presented in Section~\ref{sec:ex}. Using machine learning methods, in Section~\ref{sec:alg-and-num-res} we provide numerical solutions of the example presented in Section~\ref{sec:ex}. Finally, in Section~\ref{sec:conlcusion}, we outline some possible research directions and open problems.

\section{Time-inconsistent Markovian control problem with model uncertainty}\label{sec:robust}

In this section we state the  underlying time-inconsistent stochastic control problem.
Let  $(\Omega, \sF)$ be a measurable space,  $T\in \bN$ be a fixed time horizon, and let us denote by $\cT:=\set{0,1,2,\ldots,T}$ and $\cT':=\set{0,1,2,\ldots,T-1}$. In what follows, we implicitly assume that all considered probabilities are defined on $(\Omega,\sF)$, and as usual  $\bE_\bP$ will denote the expectation under a probability measure $\bP$.  We let $\boldsymbol \Theta\subset \bR^d$   be a non-empty Borel set\footnote{In general, the parameter space may be infinite dimensional, consisting for example of dynamic factors, such as deterministic functions of time or hidden Markov chains. In this study, for simplicity, we chose the parameter space to be a subset of $\bR^d$. In most applications, in order to avoid problems with constrained estimation, the parameter space is taken to be equal to the maximal relevant subset of $\bR^d$.}, which will play the role of the known parameter space. We consider a random process $Z=\{Z_t,\ t\in \cT\}$ on $(\Omega, \sF)$  taking values in $\bR^m$, and we denote by ${\mathbb {F}}=(\sF_t,t\in \cT)$ its natural filtration. We postulate that this process is observed by the controller, but the true law of $Z$  is unknown to the controller and assumed to be generated by a  probability measure belonging to a (known) parameterized family of probabilities $\mathbf{P}(\boldsymbol \Theta)=\{\bP_\theta,  \theta\in \boldsymbol \Theta\}$. For simplicity, we will write $\bE_\theta$ instead of  $\bE_{\bP_\theta}$. We denote by $\bP_{\theta^*}$ the measure generating the true law of $Z$, and thus $\theta ^*\in \boldsymbol
\Theta$ is the unknown true parameter. We will assume that $\boldsymbol \Theta \ne \{\theta^*\}$, namely we consider the case of a non-trivial model uncertainty.

We let $A\subset \bR^k$ to denote the set of control values. For technical reasons, such as the existence of measurable selectors, we assume that $A$ is finite, although tentatively all stated results can be extended to $A$ being a compact.  An admissible control process $\varphi$ is an $\bF$-adapted process, taking values in $A$, and we will denote by $\cA$ the set of all admissible control processes.

We consider a discrete time controlled dynamical system with the state process $X$ taking values in $\bR^n$ and having the dynamics
\begin{equation}\label{eq:mm}
X_{t+1}=S(X_t,\varphi_t,Z_{t+1}),\quad t\in\cT',\quad X_0=x_0 \in \bR^n,
\end{equation}
with $S\, :\, \bR^n\times A\times \bR^m\rightarrow \bR^n$ a measurable mapping, and $\varphi$ a control process. We limit ourselves to the class of Markovian strategies only.

The underlying uncertain control problem  is
\begin{equation}\label{eq:problem}
\sup_{\varphi\in \cA} \left (\bE_{\theta^*} \left[F(X_T)\right]+G\left(\bE_{\theta^*}[X_T]\right )\right ),
\end{equation}
where $F,G:\bR^n\to\bR$ are some given Borel measurable functions. Throughout, we will assume that all expectations are well-defined.
\begin{remark}
	Generally speaking, all results obtained in this paper can be extended to a more general control problem of the form
$$
\sup_{\varphi\in \cA} \left (\bE_{\theta^*} \left[F(X_T)\right]+G\left(\bE_{\theta^*}[H(X_T)]\right )\right ),
$$
where $H:\bR^n\to\bR$ and $G:\bR\to\bR$.  For example, if $H$ is a bijection, then putting $Y_t=H(X_t)$ reduces the problem to the case \eqref{eq:problem}. Otherwise, one can increase the dimension of the state process and replace $X_t$ with $(X_t,H(X_t))$. With slight loss of generality, and gain in readability, we opted to focus on \eqref{eq:problem}.
\end{remark}

\begin{example}[Mean-variance]\label{ex:MV}
A typical, and also practically important, example is the mean-variance (MV) control problem, where $X$ is a scalar valued process, and $F(x)=x-\gamma x^2$, $G(x)=\gamma x^2$, for some fixed weight $\gamma >0$. Sometimes $\gamma$ is refereed to as risk aversion parameter. In this case, the stochastic control problem at hand becomes
\begin{equation}\label{eq:MV}
\sup_{\varphi\in\cA}\left(\bE_{\theta^*}(X_T)-\gamma \Var_{\theta^*}(X_T)\right).
\end{equation}
In Section~\ref{sec:ex} we will return to this example in the context of portfolio selection problem.
\hfill$\square$
\end{example}

Problem \eqref{eq:problem}, in general, is time-inconsistent in the sense that the dynamic programming principle fails; see, for instance, the discussion in \cite{ShiCui2017} and \cite{BjoerkMurgoci2014}. Additionally, the parameter $\theta ^*$ is not known to the controller, and thus problem \eqref{eq:problem}  can not be solved as is. In the rest of the paper we address these two issues via \textit{time consistent adaptive robust sub-game perfect  approach}.
To achieve this, we first formally formulate an adaptive robust control problem corresponding to  \eqref{eq:problem} (see Section~\ref{sec:formulation}). The solution to this problem, by analogy, can be called adaptive robust pre-commitment strategy. Hence, by similar arguments to the case without model uncertainty this problem is time-inconsistent. To overcome this, we proceed with time consistent adaptive robust sub-game approach (see Section~\ref{sec:ARSG}).
Finally, the adaptive robust method requires handling a double (sup-inf) optimization problem at each time instance, yielding Bellman equations that are intrinsically multi-dimensional. For these reasons, computing the solutions of the corresponding control problem is a nontrivial task,  and we address this issue using machine-learning technique for the MV problem (see  Section~\ref{sec:alg-and-num-res}).

\subsection{Time-inconsistent adaptive robust control problem}\label{sec:formulation}

We mainly follow here the developments presented in \cite{BCCCJ2019}, with the key difference that herein the Markovian controls problems are inherently time-inconsistent, and for this reason we only use Markovian strategies.

A central building block in the adaptive robust approach is the recursive construction of confidence regions for the unknown model parameter $\theta^*$. We refer to \cite{BCC2017} for a general study of recursive constructions of (approximate) confidence regions for time homogeneous Markov chains, while in Section~\ref{sec:ex} we provide a specific  such recursive construction corresponding to the optimal portfolio selection problem under uncertainty. Here, we just postulate that the recursive algorithm for building confidence regions uses an $\bR^d$-valued and observed  process, say $(C_t,\ t\in \cT')$, satisfying the following abstract dynamics
\begin{equation}\label{eq:R}
C_{t+1}= R(t,C_t,Z_{t+1}),\quad t\in\cT',\ C_0=c_0,
\end{equation}
where $R:\cT'\times\bR^d\times\bR^m\to\bR^d$ is a deterministic measurable function, and $c_0\in\boldsymbol{\Theta}$.
Note that, given our assumptions about process $Z$, the process $C$ is $\bF$-adapted.
Usually $C_t$ is taken to be a consistent estimator of $\theta^*$.

Now, we fix a confidence level $\alpha\in (0,1),$ and for each time $t\in \cT'$, we assume that  an $(1-\alpha)$-confidence region, say $\mathbf{\Theta}_t \subset \bR^d$, for $\theta^*$, can be represented as
\begin{equation}\label{eq:CIR}
\mathbf{\Theta}_t=\tau(t,C_t),
\end{equation}
where, for each $t\in \cT'$, $\tau(t,\cdot)\, :\, \bR^d \rightarrow 2^{\mathbf{\Theta}}$ is a deterministic set valued   function, with $2^{\mathbf{\Theta}}$ denoting the set of all subsets of ${\mathbf{\Theta}}$. Note that in view of \eqref{eq:R} the construction of confidence regions given in \eqref{eq:CIR} is indeed recursive. In our construction of confidence regions, the mapping $\tau(t,\cdot)$ will be a measurable set valued function,  with compact values. The important property of the recursive confidence regions constructed in Section \ref{sec:ex} is that $\lim_{t\rightarrow \infty} \mathbf{\Theta}_t=\set{\theta^*}$, where the convergence is understood $\bP_{\theta^*}$ almost surely, and the limit is in the Hausdorff metric. This is not always the case, and in \cite{BCC2017} it is shown that under some general assumptions the convergence holds in probability. The sequence $\mathbf{\Theta}_t,\ t\in \cT'$ represents learning about $\theta^*$ based on the observed history up to time $t\in\cT$. We introduce the augmented state process $Y_t=(X_t,C_t),\ t\in\cT,$  and the augmented state space
\[
E_Y=\bR^n \times \bR^d,
\]
and we denote by ${\mathcal E}_Y$ the collection of Borel measurable sets in $E_Y$. In view of the above, the process $Y$ has the following dynamics,
\[
Y_{t+1}=\mathbf{G}(t, Y_t,\varphi_t,Z_{t+1}),\quad t\in \cT',
\]
where  $\mathbf{G}$  is the mapping  $\mathbf{G}\, :\, \cT'\times E_Y \times A \times \bR^m \rightarrow E_Y$ defined as
\begin{equation}\label{eq:T}
\mathbf{G}(t,y,a,z)=\big(S(x,a,z), R(t,c,z)\big),
\end{equation}
where $ y=(x,c)\in  E_Y$.

A control process $\varphi=(\varphi_t,\ t \in \cT')$ is called  \textit{Markovian control process} if (with a slight abuse of notation)
\[
\varphi_t = \varphi_t(Y_t) ,
\]
where (on the right hand side) $\varphi_t \, :\, E_Y \rightarrow A$, is a measurable mapping.

From now on, we constrain the set $\mathcal{A}$ of admissible control processes to the set of Markovian control processes. For any admissible control process $\varphi$ and for any $t\in\cT'$, we denote by $\varphi^t=(\varphi_k,\ k=t,\dots, T-1)$ the `$t$-tail' of $\varphi$; in particular, $\varphi^0=\varphi$. Accordingly, we denote by ${\mathcal A}^t$ the collection of $t$-tails $\varphi^t$;  thus, ${\mathcal A}^0={\mathcal A}$.

Let $\breve\psi_t:E_Y\to\mathbf{\Theta}$ be a measurable mapping (Knightian selector), and let us denote by $\breve\psi =(\breve\psi_t, \ t\in\cT')$ the sequence of such mappings, and by $\breve\psi^{t} = (\breve\psi_s, \ s=t,\ldots,T-1)$ the $t$-tail of the sequence $\psi$. The set of all sequences $\breve\psi$, and respectively $\breve\psi^{t}$, will be denoted by $\mathbf{\breve\Psi}$ and $\mathbf{\breve\Psi}^{t}$, respectively.

Similarly, we consider the measurable mappings ${\psi}_t:E_Y \to\mathbf{\Theta}$, such that ${\psi}_t(x,c)\in \tau(t,c)$. This, in particular, implies that $\psi_t(X_t,C_t)\in \mathbf{\Theta}_t$. Correspondingly, we  define the set of all such selectors as $\mathbf{\Psi}_t$, the set of all sequences of such mappings by $\mathbf{\Psi}=\mathbf{\Psi}_0\times \ldots \times \mathbf{\Psi}_{T-1}$, and the set of $t$-tails by $\mathbf{\Psi}^t$. Clearly,  $\mathbf{\Psi}\subset \mathbf{\breve\Psi}$. Moreover, $\psi^t\in\mathbf{\Psi}^t$ if and only if $\psi^t\in\mathbf{\breve\Psi}^{t}$ and $\psi_s(y_s)\in\tau(s,c_s), \ s=t,\ldots,T-1$.

Next, for each $(t,y,a,\theta)\in  \cT'\times  E_Y\times A\times \mathbf{\Theta}$,  we define a probability measure on $\mathcal{E}_Y$, given by
\begin{equation}\label{eq:QB-bar}
 Q(B\mid t,y,a,\theta)=\bP_\theta(Z_{t+1}\in \{z: \mathbf{G}(t,y,a,z)\in B\})=\bP_\theta\left(\mathbf{G}(t,y,a,Z_{t+1})\in B\right),\ B\in \mathcal{E}_Y.
\end{equation}
Throughout we assume that:
\begin{itemize}
  \item[(A1)] For every $t\in\cT$ and every $a\in A$,  the measure $Q(dy'\mid t,y,a,\theta)$  is a Borel measurable stochastic kernel with respect to $(y,\theta)$.
\end{itemize}
This assumption will be satisfied in the context of the mean-variance problem discussed in Section~\ref{sec:ex}.

Using Ionescu-Tulcea theorem (cf. \cite[Appendix B]{Baeuerle2011book}), for every $t\in\cT'$, control process $\varphi\in \cA^t$,  $\psi\in\boldsymbol{\Psi}^t$, time $t$ state $y_t\in E_Y$,
we define probability measure ${\mathbb Q}^{\varphi,\psi}_{y_t, t}$ on the  concatenated canonical space $\textsf{X}_{s=t+1}^TE_Y$ as follows
\begin{align}\label{eq:prob}
\mathbb{Q}^{\varphi, {\psi}}_{y_t,t}(B_{t+1}\times \cdots \times B_T) =
\int\limits_{B_{t+1}}\cdots \int\limits_{B_T}
\prod\limits_{u=t+1}^{T} Q( dy_u\mid u-1,y_{u-1},\varphi_{u-1}(y_{u-1}),{\psi}_{u-1}(y_{u-1})).
\end{align}
Correspondingly, we define the family of probability measures
${\mathcal Q}^{\varphi}_{y_t,t} =\{{\mathbb Q}^{\varphi,\psi}_{y_t, t},\  \psi \in \mathbf{ \Psi}^t \}. $
Here, and everywhere below, to simplify the notations, we simply write $\varphi\in\cA^t$, instead of $\varphi^t\in\cA^t$, and implicitly assume that the processes have the correct tail dimension.

Analogously we define the set $\breve{\mathcal Q}^{\varphi}_{y_0,0} =\{{\mathbb Q}^{\varphi,{\psi}}_{y_0,0},\ \psi\in \mathbf{\breve\Psi} \}$. In Remark \ref{rem-strong-robust} we provide a game oriented interpretation of mappings $\psi$ and $\breve\psi$, as strategies played by the nature seen as a Knightian adversary of the controller.

The \textit{time inconsistent strong robust control} problem is then given as:
\begin{align}\label{prob1R-strong}
  \sup_{\varphi\in\cA} \inf_{\mathbb{Q}\in {\mathcal Q}^{\varphi,\breve{\boldsymbol \Psi}}_{y_0,0}}\left (\bE_{\mathbb{Q}} \left(F(X_T)\right )+G\left(\bE_{\mathbb{Q}}(X_T)\right )\right ).
\end{align}
The corresponding   \textit{time inconsistent adaptive robust control} problem is:
\begin{align}\label{prob1R-adaptive}
  \sup_{\varphi\in\cA} \inf_{\mathbb{Q}\in {\mathcal Q}^{\varphi,{\boldsymbol \Psi}}_{y_0,0}} \left (\bE_{\mathbb{Q}} \left(F(X_T)\right )+G\left(\bE_{\mathbb{Q}}(X_T)\right )\right ).
\end{align}

\begin{remark}\label{rem-strong-robust}
The strong robust control problem is essentially a game problem between the controller and his/her Knightian adversary -- the nature, who may keep changing the dynamics of the underlying stochastic system over time.  In this game, the nature is not restricted in its choices of model dynamics, except for the requirement that $\breve{\psi}_t(Y_t) \in \mathbf{\Theta}$, and each choice is potentially based on the value of $Y_t$. On the other hand, the adaptive robust control problem is a game problem between the controller and  the nature, who, as in the case of strong robust control problem, may keep changing the dynamics of the underlying stochastic system over time. However, in the latter game, the nature is restricted in its choices of model dynamics to the effect that ${\psi}_t(Y_t)\in \tau(t,C_t)$.
\end{remark}

\section{Time consistent adaptive robust sub-game approach}\label{sec:ARSG}
In general, the dynamic programming principle proved in \cite[Section 2.2.1]{BCCCJ2019}, does not apply to problem \eqref{prob1R-adaptive}, which is the nature of the time inconsistency of this problem. In particular, the backward induction procedure (or dynamic programming principle) from \cite[Section 2.2.1]{BCCCJ2019},  can not be used, in general, to solve  problem \eqref{prob1R-adaptive}. Consequently, with no dynamic principle at hand, practically speaking such problems can not be solved numerically, especially when the number of steps is large.
Thus, instead of dealing with problem  \eqref{prob1R-adaptive} as is, we adopt the concept of sub-game perfect controls of \cite{BjoerkMurgoci2014} to our setup, which we will transform \eqref{prob1R-adaptive} into a time consistent problem that can be solved by using backward induction.


For convenience, for $\varphi\in \cA^{t},  \ y=(x,c)$, and for $\psi\in {\boldsymbol \Psi}^{t+1}$ we define
$$
{\mathcal Q}^{\varphi, \psi+}_{y,t} := \left \{{\mathbb Q}^{\varphi^{t},(\theta;\psi^{t+1})}_{y,t},\ \theta\in \tau(t,c) \right \},
$$
where $(\theta;\psi^{t+1}):=(\theta, \psi_{t+1}, \ldots, \psi_T)$. In what follows, we will use similar notation $(a;b)$ for concatenation of vectors $a,b$.

We define the time-$t$ time inconsistent criterion as
\begin{equation}
J_t(y,\varphi^t, {\psi}^t):=\bE_{\mathbb{Q}^{\varphi, {\psi}}_{y,t}}\left(F(X_T)\right ) + G\left(\bE_{\mathbb{Q}^{\varphi, {\psi}}_{y,t}}(X_T)\right ),\quad y=(x,c)\in E_Y,\ t\in\cT',
\end{equation}
and let
\[J_T(y)=F(x)+G(x),\quad y=(x,c)\in E_Y.\]

\begin{definition}\label{def:perfect}
The pair of strategies $(\widetilde \varphi, \widetilde \psi)$ is called \textit{sub-game perfect} if
\begin{equation}\label{eq:subgame}
\max_{a\in A} \inf_{ \theta \in \tau(t,c)}J_t(y,(a; \widetilde \varphi^{t+1}), (\theta; \widetilde {\psi}^{t+1}))=J_t(y, \widetilde \varphi^t, \widetilde {\psi}^t),
\end{equation}
for any $y=(x,c)\in E_Y$, and $t\in\cT'$, with the convention that $(a;\widetilde \varphi^{T})=a$, and $(\theta; \widetilde{\psi}^{T})=\theta$.
\end{definition}

\begin{remark}\label{remark:strong}
Similarly, one can define the sub-game perfect strategies for the strong robust case, by  replacing set $\tau(t,c)$ in Definition~\ref{def:perfect} with $\mathbf{\Theta}$.
Due to the imposed model assumptions, in particular Assumption~(A2) below, all obtained results hold true by similarity in the strong robust case.
\end{remark}

Please note that Definition \ref{def:perfect} is not a definition of equilibrium strategies in any game-theoretic sense, as we do not study any sort of game-theoretic equilibrium per se, even though we may think of the controller and the nature as two players. This definition is inspired by the concept of the sub-game perfect Nash equilibrium that was applied in the context of time-inconsistent control problems, as presented for example in Definition 2.2 in [BM14]. However, sub-game perfect Nash equilibrium  is not really the same as the classical game-theoretic concept of Nash equilibrium, and should not be interpreted as such.

The idea of Definition~\ref{def:perfect} is to view the problem in the embedded sequential optimization terms:  at each point of time there are two decision makers who chose decisions impacting evolution of the system: the controller and the nature.  The two decision makers acting at time $t\in\cT'$ know that all the pairs of decision makers coming after them will use the control $(\widetilde{\varphi}^{t+1},\widetilde{\psi}^{t+1})$. Given such knowledge, the two time-$t$ decision makers optimize over $A$ and $\tau(t,c)$, respectively, so generate their decisions. Following up on item 2) from the Introduction, we might interpret this as the game played by the two players,  the controller and the nature, at time-$t$, with their future selves.

Throughout, we will use the notation
$$
V_t(y):=J_t(y, \widetilde \varphi^t, \widetilde {\psi}^t),
$$
for $y=(x,c)\in E_Y, \ t\in\cT'$ and some $(\widetilde \varphi^t, \widetilde {\psi}^t)$ pair of sub-game perfect strategy. We remark that  due to the Definition~\ref{def:perfect} and Theorem~\ref{thm:1} proved below, the value of $V_t(y)$ does not depend on the choice of the sub-game perfect strategy. We will also show that there exists a sub-game perfect pair of strategies to the original stochastic control problem \eqref{prob1R-adaptive}, and thus $V_t(y)$ is well-defined.

Note that, for any $\varphi^{t}\in \cA^{t}, \ \psi^{t+1}\in\boldsymbol\Psi^{t+1}$ we have that
\[
\inf_{ \theta \in \tau(t,c)}J_t(y,\varphi^{t}, (\theta; {\psi}^{t+1})) =
\inf_{ {\mathbb{Q}} \in {\mathcal Q}^{\varphi, \psi+}_{y,t}} \left ( \bE_{\mathbb{Q}}\left (F(X_T)\right )+G\left(\bE_{\mathbb{Q}}(X_T)\right )\right ).
\]
For $y=(x,c)\in E_Y$, $t\in\cT'$, $\varphi\in A^t$, and $\psi\in\boldsymbol{\Psi}^t$ we denote
\begin{align*}
f_t^{{\varphi, \psi}}(y) & :=\bE_{\mathbb{Q}^{\varphi, {\psi}}_{y,t}}F(X_T), \quad
g_t^{{\varphi, \psi}}(y)  :=\bE_{\mathbb{Q}^{\varphi, {\psi}}_{y,t}}(X_T),\\
f_{T}^{{\varphi, \psi}}(y) & := F(x), \quad  \ g_{T}^{{\varphi, \psi}}(y):=x,
\end{align*}
and for $y=(x,c)\in E_Y$, and $t\in\cT$, we put $\widetilde g_t(y)=g_t^{{ \widetilde \varphi, \widetilde \psi}}(y)$.
In addition, we define the following integral operator,
\[
Q^{a,\theta}_{y,t}f:=\int_{E_Y}f(y')Q(dy'|t,y,a,\theta),
\]
where $Q$ is given by \eqref{eq:QB-bar}.

Clearly,
\[
f_t^{{\varphi, \psi}}(y)= Q^{{\varphi_t, \psi_t}}_{y,t}f_{t+1}^{\varphi, {\psi}}, \quad
g_t^{{\varphi, \psi}}(y)= Q^{{\varphi_t, \psi_t}}_{y,t}g_{t+1}^{\varphi, {\psi}}.
\]

With this at hand, we have the following counterpart of \cite[Lemma~3.2]{BjoerkMurgoci2014}.

\begin{lemma}\label{lem:lemma1}
For $y=(x,c)\in E_Y,\ a\in A, \theta\in \tau(t,c)$, and $t\in\cT'$, the following identity holds
$$
J_t(y,(a;\varphi^{t+1}), (\theta;{\psi}^{t+1}))=Q^{a,\theta}_{y,t}J_{t+1}(\cdot,\varphi^{t+1}, {\psi}^{t+1})
-\left[Q^{a,\theta}_{y,t}(G\circ g_{t+1}^{^{\varphi, \psi}})- G(Q^{a,\theta}_{y,t}g_{t+1}^{{\varphi, \psi}})\right ],
$$
with the convention that
\[
J_{T}(y,\varphi^{T}, {\psi}^{T})=J_T(y)=F(x)+G(x),\quad y=(x,c)\in E_Y.
\]
\end{lemma}

\begin{proof}
We have
\[
Q^{a,\theta}_{y,t}J_{t+1}(\cdot,\varphi, {\psi})=Q^{a,\theta}_{y,t}f_{t+1}^{\varphi, {\psi}}
+Q^{a,\theta}_{y,t} \left(G \circ g_{t+1}^{\varphi, \psi} \right)
\]
and
\begin{align*}
J_t(y,(a;\varphi^{t+1}), (\theta;{\psi}^{t+1}))
& =f_t^{(a;\varphi^{t+1}), (\theta;{\psi}^{t+1})}(y)+G\left (g_t^{(a;\varphi^{t+1}), (\theta;{\psi}^{t+1})}(y)\right ) \\
&=Q^{a,\theta}_{y,t}f_{t+1}^{\varphi, \psi} + G \left(Q^{a,\theta}_{y,t} g_{t+1}^{\varphi, \psi}\right).
\end{align*}
This proves the result.

\end{proof}

Now we are in the position to prove the first main result about the backward recursion for $V_t$ in terms of the corresponding Bellman equations.

\begin{theorem}\label{thm:1}
A pair $(\wt \varphi, \wt \psi)$ of Markovian strategies is a pair of sub-game  perfect strategies if and only if
\begin{align}
V_t(y) & =\max_{a\in A} \inf_{ \theta \in \tau(t,c)}\left(Q^{a,\theta}_{y,t} V_{t+1}- \left[Q^{a, \theta}_{y,t} (G\circ \widetilde g_{t+1})- G(Q^{a,\theta}_{y,t} \widetilde g_{t+1})\right ]\right ),\label{eq:Bel1} \\
\widetilde g_{t}(y) & = Q^{ \widetilde \varphi_t, \widetilde \psi_t}_{y,t} \widetilde g_{t+1}, \label{eq:Bel2}\\
V_T(y) & = F(x) + G(x),\label{eq:Bel3}
\end{align}
for any $y =(x,c)\in E_Y$, and $t\in \cT'$.
\end{theorem}

\begin{proof} $(\Rightarrow)$ In view of Lemma~\ref{lem:lemma1}, we have that
\begin{align*}
J_t(y,(a;\widetilde \varphi^{t+1}), (\theta;\widetilde {\psi}^{t+1}))
& =Q^{a,\theta}_{y,t}J_{t+1}(\, \cdot\, ,\widetilde \varphi^{t+1}, \widetilde {\psi}^{t+1})
-\left[Q^{a,\theta}_{y,t} (G \circ g_{t+1}^{\widetilde \varphi, \widetilde \psi} )- G(Q^{a,\theta}_{y,t}g_{t+1}^{{\widetilde \varphi, \widetilde \psi}})\right ] \\
& =Q^{a,\theta}_{y,t}V_{t+1}-\left[Q^{a,  \theta}_{y,t}(G \circ \widetilde g_{t+1})- G(Q^{a,\theta}_{y,t}\widetilde g_{t+1})\right ], \\
& \qquad y=(x,c)\in E_Y,\ a\in A, \ \theta\in \tau(t,c),\ t\in\cT'.
\end{align*}
Thus,
\begin{align*}
V_t(y) & = \max_{a\in A} \inf_{ \theta \in \tau(t,c)}J_t(y,(a;\widetilde \varphi^{t+1}), (\theta;\widetilde {\psi}^{t+1})) \\
& =\max_{a\in A} \inf_{ \theta \in \tau(t,c)}\left(Q^{a,\theta}_{y,t}V_{t+1}-\left[Q^{a, \theta}_{y,t}(G \circ \widetilde g_{t+1})- G(Q^{a,\theta}_{y,t}\widetilde g_{t+1})\right ]\right ), \\
& \qquad \qquad y=(x,c)\in E_Y,\ a\in A, \ \theta\in \tau(t,c),\ t\in\cT'.
\end{align*}

\noindent $(\Leftarrow)$ We start with $t=T-1$. Note that $V_T\equiv J_T$. Thus, for $y=(x,c)\in E_Y$,  we have
\begin{align*}
V_{T-1}(y) & =\max_{a\in A} \inf_{ \theta \in \tau(T-1,c)}\left(Q^{a,\theta}_{y,T-1}V_{T}-\left[Q^{a, \theta}_{y,T-1} (G\circ \check g_{T} ) - G(Q^{a,\theta}_{y,T-1}\check g_{T})\right ]\right )\\
& =Q^{\check \varphi_{T-1}(y),\check \psi_{T-1}(y)}_{y,T-1}V_{T}-\left[Q^{\check \varphi_{T-1}(y), \check \psi_{T-1}(y)}_{y,T-1}(G\circ \check g_{T})- G(Q^{\check \varphi_{T-1}(y),\check \psi_{T-1}(y)}_{y,T-1}\check g_{T})\right].
\end{align*}
Using Lemma~\ref{lem:lemma1}, we deduce
\begin{equation}\label{eq:T-1}
V_{T-1}(y)=\max_{a\in A} \inf_{ \theta \in \tau(T-1,c)}J_{T-1}(y,a, \theta)=J_{T-1}(y,\check \varphi_{T-1}, \check \psi_{T-1})=J_{T-1}(y,\check \varphi^{T-1}, \check \psi^{T-1}),
\end{equation}
which verifies \eqref{eq:subgame} for $t=T-1$ and $\widetilde \varphi_{T-1}\equiv \check \varphi_{T-1}$ and $\widetilde \psi_{T-1}\equiv \check \psi_{T-1}$.

Next, we let $t=T-2$. Then, using \eqref{eq:T-1} we produce
\begin{align*}
V_{T-2}(y) &  =\max_{a\in A} \inf_{ \theta \in \tau(T-2,c)}\left(Q^{a,\theta}_{y,T-2}V_{T-1}-\left[Q^{a, \theta}_{y,T-2}(G\circ\check g_{T-1})- G(Q^{a,\theta}_{y,T-2}\check g_{T-1})\right ]\right )\\
& =Q^{\check \varphi_{T-2}(y),\check \psi_{T-2}(y)}_{y,T-2}V_{T-1}-\left[Q^{\check \varphi_{T-2}(y), \check \psi_{T-2}(y)}_{y,T-2}(G\circ\check g_{T-1})- G(Q^{\check \varphi_{T-2}(y),\check \psi_{T-2}(y)}_{y,T-2}\check g_{T-1})\right ].
\end{align*}
Again, in view of Lemma \ref{lem:lemma1} we obtain
\begin{align}
V_{T-2}(y) & =\max_{a\in A} \inf_{ \theta \in \tau(T-2,c)}J_{T-2}(y,(a;\check \varphi_{T-1}), (\theta; \check \psi_{T-2})) \\
& =J_{T-2}(y,(\check \varphi_{T-2};\check \varphi_{T-1}), (\check \psi_{T-2};\check \psi_{T-1}))=J_{T-1}(y,\check \varphi^{T-2}, \check \psi^{T-2}), \label{eq:T-2}
\end{align}
which verifies \eqref{eq:subgame} for $t=T-2$ and $\widetilde \varphi^{T-2}\equiv \check \varphi^{T-2}$ and $\widetilde \psi^{T-2}\equiv \check \psi^{T-2}$.
Proceeding in the analogous way for $t=T-3,\ldots,0$ we complete the proof.
\end{proof}

\subsection{Existence of sub-game perfect strategies}\label{sec:existence}
In this section we study the existence of a pair of sub-game perfect strategies.  To this end, in addition the model assumptions from Section~\ref{sec:robust} and the Assumption~(A1), we make the following standing assumptions:
\begin{itemize}
\item[(A2)] The set $\boldsymbol\Theta$ is a compact subset of $\bR^d$.
\item[(A3)] The probability measures in the family $\Set{ Q(\,\cdot\mid t,y,a,\theta), \ t\in\cT, \  y \in E_Y, a \in A, \ \theta\in\boldsymbol\Theta}$ are equivalent.
\end{itemize}
We will show that these assumptions are satisfied in the example studied in Section \ref{sec:ex}. We note that Assumption (A3) could be alternatively formulated in terms of the probability measures generated by $Z^\theta, \ \theta\in\boldsymbol\Theta$.


Next we give the main result of this section.

\begin{theorem}\label{th:existence}
The functions $V_t$, $t\in\cT$, are lower semi-analytic (l.s.a.), and there exists a pair of sub-game perfect strategies.
\end{theorem}

\begin{proof} We will prove existence of sub-game perfect strategies by applying  Theorem~\ref{thm:1} and show, by backward induction, that for any $t\in\cT'$, $y\in E_Y$, there exist universally measurable $\widetilde \varphi_t(y)$ and $\widetilde \psi_t(y, \widetilde\varphi_t(y))$ such that
\begin{align}\label{eq:meas-select}
V_t(y) & = Q^{\widetilde\varphi_t, \widetilde\psi_t}_{y,t}V_{t+1} - Q^{\widetilde\varphi_t, \widetilde \psi_t}_{y,t}(G\circ \widetilde g_{t+1})
+G\left(Q^{\widetilde{\varphi}_t,\widetilde\psi_t}_{y,t}\widetilde g_{t+1}\right),
\end{align}
and $V_t(y)$ is l.s.a., with $\widetilde {g}_{t+1}=g^{\widetilde {\varphi},\widetilde {\psi}}_{t+1}$, and where we recall that  $V_t(y):=J_t(y, \bar \varphi^t, \bar{\psi}^t)$ for some pair of sub-perfect strategies.

In view of Lemma~\ref{lem:lemma1}, we have that for $t=T-1$
\begin{align*}
 V_{T-1}(y)=
 \max_{a\in A}\inf_{\theta\in\tau(T-1,c)} \Big(Q^{a,\theta}_{y,T-1}V_T- Q^{a,\theta}_{y,T-1}
 (G\circ \widetilde {g}_T) + G\left(Q^{a,\theta}_{y,T-1}\widetilde {g}_T\right)\Big),
\end{align*}
where we put  $\widetilde{g}_T(y)=x$  which is Borel measurable in $y$.

Hence, according to our assumptions, $G\circ \widetilde {g}_T$ and $V_T(y)=F(x)+G(x)$ are also Borel measurable. By Assumption (A1), and using \cite[Proposition~7.29]{BertsekasShreve1978Book}, the following functions
$$
Q^{a,\theta}_{y,T-1}V_T, \quad Q^{a,\theta}_{y,T-1}(G \circ \widetilde {g}_T), \quad G\left(Q^{a,\theta}_{y,T-1}\widetilde {g}_T\right)
$$
are Borel measurable in $(y,a,\theta)$.  Therefore, the function
\begin{align*}
\check{V}_{T-1}(y,a,\theta):=Q^{a,\theta}_{y,T-1}V_T - Q^{a,\theta}_{y,T-1}(G\circ \widetilde {g}_T)
 + G\left(Q^{a,\theta}_{y,T-1}\widetilde{g}_T\right)
\end{align*}
is Borel measurable. Moreover, for any $b\in\bR$, the set $\{(y,a,\theta)\in E_Y\times A\times\tau(T-1,c): \check{V}_{T-1}(y,a,\theta)<b)\}$ is a Borel measurable subset of $E_Y\times A\times\tau(T-1,c)$. Since $E_Y\times A\times\tau(T-1,c)$ is a closed subset of $E_Y\times A\times\boldsymbol\Theta$, which is a Polish space (and thus a Borel space), then it is a Borel subspace. In turn, by \cite[Proposition~7.36]{BertsekasShreve1978Book}, the set $\{(y,a,\theta)\in E_Y\times A\times\tau(T-1,c): \check{V}_{T-1}(y,a,\theta)<b)\}$ is analytic. Consequently, the function $\check{V}_{T-1}(y,a,\theta)$ is l.s.a..

By adopting the notations in \cite[Proposition~7.50]{BertsekasShreve1978Book}, we write\footnote{{The notation $\mathrm{X}$ and $\mathrm{Y}$, representing the relevant Borel spaces, should not be confused with the notation $X$ and $Y$ representing the relevant processes.}}
  \begin{align*}
  	\mathrm{X}&=E_Y\times A=\bR^n\times\bR^d\times A,\quad \mathrm{x}=(y,a),\\
  	\mathrm{Y}&=\boldsymbol\Theta,\quad \mathrm{y}=\theta,\\
  	\mathrm{D}&=\bigcup_{(y,a)\in E_Y\times A}\{(y,a)\}\times\tau(T-1,c),\\
   	\mathrm{f}(\mathrm{x},\mathrm{y})&=\check{V}_{T-1}(y,a,\theta).
  \end{align*}
Recall that in view of the prior assumptions, $\mathrm{X}$ and $\mathrm{Y}$ are both Borel spaces. $\mathrm{D}$ is a closed set and therefore analytic, and the cross section $\mathrm{D}_{\mathrm{x}}=\mathrm{D}_{(y,a)}=\{\theta\in\boldsymbol\Theta:(y,a,\theta)\in\mathrm{D}\}$ is given by $\mathrm{D}_{(y,a)} = \tau(T-1,c)$.
  Hence, by \cite[Proposition~7.47]{BertsekasShreve1978Book}, the function
  $$
  	\check{V}^*_{T-1}(y,a)=\inf_{\theta\in\tau(T-1,c)}\check{V}_{T-1}(y,a,\theta), \quad (y,a)\in E_Y\times A,
  $$
  is l.s.a..
  Moreover, in view of \cite[Proposition~7.50]{BertsekasShreve1978Book}, for any $\varepsilon>0$, there exists an analytically measurable function $\widetilde{\psi}^{\varepsilon}_{T-1}(y,a)$ such that
  \begin{align*}
  	\check{V}_{T-1}(y,a, \widetilde{\psi}^{\varepsilon}_{T-1}(y,a)) \leq
  	\begin{cases}
  	\check{V}^*_{T-1}(y,a)+\varepsilon, \quad &\text{if}\ \check{V}^*_{T-1}(y,a)>-\infty,\\
  	-\frac{1}{\varepsilon}, \quad &\text{if}\ \check{V}^*_{T-1}(y,a)=-\infty.
  	\end{cases}
  \end{align*}
Therefore, for any fixed $(y,a)$, we obtain a sequence $\set{\widetilde {\psi}^{\frac{1}{n}}_{T-1}(y,a), \ n\in\bN$}, such that
$$
\lim_{n\to\infty}\check{V}_{T-1}(y,a, \widetilde {\psi}^{\frac{1}{n}}_{T-1}(y,a))=\check{V}^*_{T-1}(y,a).
$$
By Assumption (A2), there exists a convergent subsequence $\set{\widetilde {\psi}^{\frac{1}{n_k}}_{T-1}(y,a), k\in\bN}$, such that its limit $\widetilde {\psi}^*_{T-1}(y,a)$ satisfies
$$
  \check{V}_{T-1}(y,a,\widetilde{\psi}^*_{T-1}(y,a))=\check{V}^*_{T-1}(y,a).
$$
Clearly, $V_{T-1}(y)=\max_{a\in A}\check{V}^*_{T-1}(y,a)$. Next, for every fixed $a\in A$, any $b\in\bR$, we write the following complement of the upper level set as
  \begin{align*}
  	\{y\in E_Y: \check{V}^*_{T-1}(y,a)<b\}&=\{y\in E_Y: (y,a)\in \check{V}^{*,-1}_{T-1}((-\infty,b))\}\\
  &=\text{proj}_{E_Y}\left \{\check{V}^{*,-1}_{T-1}((-\infty,b))\cap(E_Y\times \{a\})\right \}.
  \end{align*}
  As a projection of an analytic set, such set is analytic, and moreover, $\check{V}^*_{T-1}(y,a)$ is l.s.a. in $y$ for every $a\in A$.
  Thus, we get that $V_{T-1}(y)$ is l.s.a. as being the maximum of a finite collection of l.s.a. functions.

  For every $y\in E_Y$, define $\widetilde{\varphi}^*_{T-1}(y)=\text{argmax}_{a\in A}\{\check{V}^*_{T-1}(y,a)=V_{T-1}(y)\}$.
  Note that the set $A$ is finite.
  Hence, we write
  $$
  \{y\in E_Y:\check{V}^*_{T-1}(y,a)=V_{T-1}(y)\}=\{y\in E_Y:\check{V}^*_{T-1}(y,a)-V_{T-1}(y)=0\}.
  $$
  Since $\check{V}^*_{T-1}(y,a)$ is l.s.a. in $(y,a)$, then it is analytically measurable and universally measurable in $(y,a)$.
  Moreover, it is universally measurable in $y$ for every $a$.
  Similarly, the function $V_{T-1}(y)$ is universally measurable in $y$ as well.
  We get that the set $\{y\in E_Y:\check{V}^*_{T-1}(y,a)\}$ is universally measurable for every $a\in A$.
  Thus, the function $\widetilde{\varphi}^*_{T-1}(y)$ is universally measurable and it is the optimal selector.
It is also straightforward to verify \eqref{eq:meas-select}, and hence the proof for $t=T-1$ is complete.

Next, we note that the stochastic kernel $Q(dy'|T-1,y, \widetilde {\varphi}_{T-1}, \widetilde{\psi}_{T-1})$ is universally measurable as it is a composition of Borel measurable and universally measurable mappings (cf. \cite[Proposition~7.44]{BertsekasShreve1978Book})
Hence, by \cite[Proposition~7.46]{BertsekasShreve1978Book}, we deduce that $\widetilde {g}_{T-1}(y)=Q^{ \widetilde{\varphi}_{T-1},\widetilde{\psi}_{T-1}}_{y,T-1}\widetilde{g}_T$ is universally measurable.

For $0<t\leq T-1$, assume that $V_t(y)$ is l.s.a. and $\widetilde{g}_t(y)$ is universally measurable.
Then, by \cite[Lemma~7.27]{BertsekasShreve1978Book}, for any chosen $\theta\in\boldsymbol\Theta$, we have that there exists a Borel measurable function $\check{g}_t(y)$ such that $\check {g}_t(y)=\widetilde{g}_t(y)$ almost surely under  the reference measure $\bP$.
Consequently, by Assumption (A3) we have $Q^{a,\theta}_{y,t}\check f=Q^{a,\theta}_{y,t}\widetilde f$ for any two integrable functions $\check f$ and $\tilde f$, such that $\check f$ is Borel measurable, $\widetilde f$ is universally measurable, and $\check f=\tilde f$ almost surely under $\bP$. Thus, $Q^{a,\theta}_{y,t}\check{g}_t=Q^{a,\theta}_{y,t}\widetilde{g}_t$.

Finally, we note that the stochastic kernel $Q(dy'|t,y,a,\theta)$ is Borel measurable in $(y,a,\theta)$. By \cite[Proposition~7.48]{BertsekasShreve1978Book}, it implies that $Q^{a,\theta}_{y,t-1}V_t$ is l.s.a.. On the other hand, since $G\circ\tilde{g}_t$ is Borel measurable, we have that $-Q^{a,\theta}_{y,t}(G\circ\tilde{g}_t)$ and $G\left(Q^{a,\theta}_{y,t}\tilde{g}_t\right)$ are also Borel measurable. Thus, they are also l.s.a..
The rest of the proof follows analogously. By induction, we conclude that \eqref{eq:meas-select} holds true for any $t\in\cT'$, $y\in E_Y$, and an universally measurable pair sub-game perfect strategies exist.
\end{proof}

\section{Uncertain dynamic mean-variance portfolio selection problem}\label{sec:ex}

In this section we will present an example that illustrates the results of Section~\ref{sec:ARSG}. Namely, we consider a dynamic mean-variance portfolio selection problem, when an investor is deciding at time $t$ on investing in a risky asset and a risk-free banking account in order to maximize the terminal weighted mean-variance criterion of the form \eqref{eq:MV}, subject to market model uncertainty.

We take a risk-free asset $B$ with a constant interest rate $r=(B_{t+1}-B_t)/B_t$, and a risky asset, say a stock, with the corresponding return from time $t$ to $t+1$ denoted by $ r^s_{t+1}$. We assume that the return process $r^s$, is observed.
The dynamics of the wealth process, say $W$, produced by a self-financing trading strategy is given by
\begin{equation}\label{eq:wealthEx1}
  W_{t+1} = W_t(1+r + \varphi_t (r^s_{t+1} - r)),\quad t\in {\cal T}',
\end{equation}
with the initial wealth $W_0=w_0>0$, and where $\varphi_t$ denotes the proportion of the portfolio wealth invested in the risky asset from time $t$ to $t+1$. We assume that the process $\varphi$ takes finitely many values, say $a_i,\ i=1,\ldots,N$, where $a_i\in [0,1]$.

We further assume that $r_t^s + 1, \, t=1,\ldots,T-1$, is an i.i.d. sequence of log-normal distributed random variables, or saying differently we assume that the excess log-returns are normally distributed. Namely,
$$
r_t^s = e^{Z_t} - 1,
$$
where $Z_t$ is an i.i.d. sequence of Gaussian random variables with mean $\mu$ and variance $\sigma^2$. Equivalently, we put
$Z_{t} = \mu +\sigma \varepsilon_{t}$, where $\varepsilon_t, \ t\in \cT'$ are i.i.d. standard Gaussian random variables. Note that under the above model assumptions, the wealth process remains positive a.s. The model uncertainty will come from the unknown parameters $\mu$ and/or $\sigma$. Using the notations from Section~\ref{sec:robust}, here we have that $X_t=W_t$, and setting $x=w$ we get
$$
S(w,a,z)=w(1+r+a(e^z-1-r)),  \quad A=\set{a_i,\ i=1,\ldots,N}.
$$
Same as in Example~\ref{ex:MV}, we take $F(w)=w-\gamma w^2$, and $G(w)=\gamma w^2$.
Formally, the investor's adaptive robust mean variance problem is formulated as follows
\begin{equation}\label{eq:ex1Problem}
  \sup_{\varphi \in\cA} \inf_{\bQ\in\cQ_{y_0,0}^{\varphi,\boldsymbol\Psi}}\left (\bE_{\bQ}(W_T)-\gamma \Var_{\bQ}(W_T)\right ),
\end{equation}
where $\cA$ is the set of Markovian trading strategies. We will find a pair of sub-game perfect strategy corresponding to \eqref{eq:ex1Problem}.

We will discuss two cases: Case 1 - unknown mean $\mu$ and known standard deviation $\sigma$, and Case II - both $\mu$ and $\sigma$ are unknown.

\smallskip\noindent
\textit{\textbf{Case I.}} Assume that $\sigma$ is known, and the model ambiguity comes only from the  parameter $\mu$, whose true but unknown value is denoted by $\mu^*$.
Thus, using the notations from Section~\ref{sec:robust}, we have that $\theta^* = \mu^*,$ $\theta = \mu,$ and we take $C_t=\widehat \mu_t$, $\boldsymbol\Theta =[\underline{\mu},\overline{\mu}]\subset\bR$, where $\widehat{\mu}$ is  a point estimator of $\mu$, given the observations of process $Z$, that takes values in $\boldsymbol{\Theta}$. The values of the boundaries $\underline{\mu}$ and $\overline{\mu}$ are fixed a priori by the observer. For the detailed discussion regarding  the construction of such estimators we refer to \cite{BCC2017}.
For this example, it is enough to take as $\widehat{\mu}$ the Maximum Likelihood Estimator (MLE), which is the sample mean in this case, projected appropriately on $\boldsymbol\Theta$. Formally, the recursion construction of $\widehat{\mu}$ is defined as follows:
\begin{equation}\label{eq:esti_proj}
\begin{aligned}
    \widetilde{\mu}_{t+1}&=\frac{t}{t+1}\widehat{\mu}_{t}+{\frac{1}{t+1}} Z_{t+1},\\
    \widehat{\mu}_{t+1}&= \pi(\widetilde{\mu}_{t+1}), \qquad t\in\cT',
\end{aligned}
\end{equation}
with $\widehat{\mu}_0=c_0\in \boldsymbol{\Theta}$, and where $\pi$ is the projection to the closest point in $\boldsymbol\Theta$, namely
$\pi(\mu) = \mu$ if $\mu\in[\underline{\mu},\overline{\mu}]$, $\pi(\mu) = \underline{\mu}$ if $\mu<\underline{\mu}$, and $\pi(\mu) = \overline{\mu}$ if $\mu> \overline{\mu}$. We take as the initial guess $c_0$ any point in $\boldsymbol\Theta$.

Putting the above together we get that the function $\mathbf{G}$ defined in \eqref{eq:T} is given here by

\begin{equation}\label{G-mu}
\mathbf{G}(t,w,c,a,z) = \left(w(1+r+a(e^z-1-r)), \pi \left (\frac{t}{t+1}c+\frac{1}{t+1}z\right )\right).
\end{equation}

It can be easily verified that the kernel $Q(\,\cdot\,|t,y,a,\mu)$, defined in terms of function $\mathbf{G}$ given in \eqref{G-mu}, satisfies  Assumption (A1), for example by using \cite[Proposition~7.26]{BertsekasShreve1978Book}. Obviously Assumption (A2) is satisfied.


As far as Assumption (A3), let $B\in\cE_Y$ such that $Q(B\mid t, y, a, \mu)=0$. In view of \eqref{eq:QB-bar} we have that
$$
\bP_{\mu}(Z_{t+1}\in\{z:\mathbf{G}(t,y,a,z)\in B\})=0,
$$
where $Z_{t+1}\sim N(\mu,\sigma^2)$. Due to the normality, it is clear that for any $\mu'\in\boldsymbol\Theta$, we also have
$$
\bP_{\mu'}(Z'_{t+1}\in\{z:\mathbf{G}(t,y,a,z)\in B\})=0
$$
with $Z'_{t+1}\sim N(\mu',\sigma^2)$. Hence, $Q(B\mid t, y, a, \mu')=0$, and thus the stochastic kernels $Q(\cdot\mid t,y,a,\mu)$ and $Q(\cdot\mid t,y,a,\mu')$ are equivalent and Assumption~(A3) is fulfilled.

Now, we note that the $(1-\alpha)$-confidence region  for $\mu^*$ at time $t$ is given as
\[
 \mathbf{\Theta}_t=\tau(t,\widehat \mu_t),
\]
where
 \begin{equation}\label{tau-mu}
\tau(t,c)=\left [\max\left(c-\frac{\sigma}{\sqrt t}q_{\alpha/2},\underline{\mu}\right), \min\left(c +\frac{\sigma}{\sqrt t}q_{\alpha/2},\overline{\mu}\right)\right ],
 \end{equation}
and  where  $q_\alpha$ denotes the $\alpha$-quantile  of a standard normal distribution.
We take closed intervals in \eqref{tau-mu} to preserve compactness. With these at hand we construct the kernel $Q$ according to \eqref{eq:QB-bar}, and the set of probability measures $\cQ^{\varphi,\boldsymbol\Psi}_{y_0,0}$ on canonical space according to \eqref{eq:prob}. We recall that in the present case $y_0=(w_0,c_0)$.

The Bellman equations \eqref{eq:Bel1}--\eqref{eq:Bel3} take the form
\begin{align}
V_t(y)&=\max_{a\in A} \inf_{ \theta \in \tau(t,c)}\left(Q^{a,\theta}_{y,t}V_{t+1}-\left[Q^{a, \theta}_{y,t}(\gamma \hat g^2_{t+1}(\cdot))- \gamma (Q^{a,\theta}_{y,t}\hat g_{t+1})^2\right ]\right ), \label{Bel1-MV}\\
\hat g_{t}(y)&=Q^{\hat \varphi_t,\hat \psi_t}_{y,t}\hat g_{t+1},\label{Bel2-MV}\\
V_T(y)&=w,\label{Bel3-MV}\\
y&=(w,c)\in E_Y,\ t\in\cT', \nonumber
\end{align}
with $\tau(t,c)$ given in \eqref{tau-mu}.

In view of  Theorem \ref{thm:1}  a pair $(\hat \varphi,\hat \psi)$ of Markov strategies satisfying \eqref{Bel1-MV}--\eqref{Bel3-MV} is a pair of sub-game perfect strategies for the adaptive robust mean-variance problem \eqref{eq:ex1Problem} with  unknown $\mu$.

In the next section we will solve the equations \eqref{Bel1-MV}--\eqref{Bel3-MV} for a pair $(\hat \varphi,\hat \psi)$ using a machine learning based method. Note that although the dimension of the state space $E_Y$ is two in the present case, which allows for efficient use of the traditional grid-based method to numerically solve the Bellman equations, our machine learning based method, originally proposed in \cite{ChenLudkovski2019} can be applied to high dimensional problems where gridding is extremely inefficient. Generally speaking, this approach overcomes the challenges met in high dimensional (robust) stochastic control problems.

\smallskip\noindent
\textit{\textbf{Case II.}} Here we assume that both $\mu$ and $\sigma$ are unknown, and thus, in the notations of Section~\ref{sec:robust}, we have $\theta^*=(\mu^*,(\sigma^*)^2)$, $\theta=(\mu,\sigma^2)$, $\boldsymbol\Theta= [\underline{\mu},\overline{\mu}]\times[\underline{\sigma}^2,\overline{\sigma}^2]\subset\bR\times\bR_+$, for some fixed $\underline{\mu},\overline{\mu}\in\bR$ and $\underline{\sigma}^2, \overline{\sigma}^2\in\bR_+$.
Similar to the Case~I, we take as the point estimators for $\mu^*$ and $(\sigma^*)^2$ the corresponding  MLEs, namely the sample mean and respectively the sample variance, projected appropriately to the rectangle $\boldsymbol\Theta$.  It is shown in \cite{BCC2017}  that the following recursions hold true
\begin{align*}
   \widetilde{\mu}_{t+1} &=\frac{t}{t+1}\widehat{\mu}_{t}+\frac{1}{t+1} Z_{t+1}, \\
    \widetilde{\sigma}_{t+1}^2 &=\frac{t}{t+1}\widehat{\sigma}^2_{t}+\frac{t}{(t+1)^{2}}(\widehat{\mu}_{t}- Z_{t+1})^2,\\
    (\widehat{\mu}_{t+1},\widehat{\sigma}_{t+1}^2)&=\pi(\widetilde{\mu}_{t+1},\widetilde{\sigma}_{t+1}^2), \quad t=1,\ldots,T-1,
  \end{align*}
with some initial guess $\widehat{\mu}_0=c_0'$, and $\widehat{\sigma}_0^2=c_0''$, and where $\pi$ is the projection\footnote{We refer to \cite{BCC2017} for precise definition of the projection $\pi$, but essentially it is defined as the closest point in the set $\boldsymbol\Theta$.} defined similarly as in \eqref{eq:esti_proj}.
Consequently, we set $C_t=(C_t',C_t'')=(\widehat\mu_t, \widehat\sigma^2_t), t\in\cT$, and respectively we have
\[
R(t,c,z)=\pi\left(\frac{t}{t+1}c'+\frac{1}{t+1}z,\frac{t}{t+1}c''+\frac{t}{(t+1)^{2}}(c'-z)^2\right),
\]
with $c=(c',c'')$. Thus, in this case, we have

\begin{equation}\label{G-mu-sigma}
\mathbf{G}(t,v,c,a,z) = \left(v(1+r+a(e^z-1-r)), \pi\left (\frac{t}{t+1}c'+\frac{1}{t+1}z, \frac{t}{t+1}c''+\frac{t}{(t+1)^{2}}(c'-z)^2\right ) \right).
\end{equation}

Similarly as in Case I with regard to function $\mathbf{G}$ given in \eqref{G-mu}, it can be easily verified that the kernel $Q(\,\cdot \mid t,y,a,\mu)$, defined in terms of function $\mathbf{G}$ given in \eqref{G-mu-sigma}, satisfies  Assumptions (A1) and (A3).  The $(1-\alpha)$-confidence region  for $(\mu^*,(\sigma^*)^2)$ at time $t$ is  the ellipsoid given by
\begin{align*}
\mathbf{\Theta}_t=\tau(t,\widehat \mu_t,\widehat \sigma^2_t),
\end{align*}
with
%
\begin{equation}\label{eq:tau-mu-sigma}
\tau(t,c)=\left\{ (\mu,\sigma^2)\in\boldsymbol\Theta \ : \ \frac{t}{c''}(c'-\mu)^2 + \frac{t}{2 (c'')^2}(c'' -\sigma^{2})^2\leq \kappa\right\},
\end{equation}
where $\kappa$ is the $(1-\alpha)$ quantile of the $\chi^2$ distribution with two degrees of freedom. Accordingly, equations \eqref{eq:Bel1}--\eqref{eq:Bel3} take the form analogous to \eqref{Bel1-MV}--\eqref{Bel3-MV} with, in particular,  $\tau(t,c)$ given in \eqref{eq:tau-mu-sigma}.
In view of Theorem \ref{thm:1} a pair $(\hat \varphi,\hat \psi)$ of Markov strategies satisfying such equations  is a pair of sub-game perfect strategies for the adaptive robust mean-variance problem \eqref{eq:ex1Problem} with  unknown $\mu$ and $\sigma$. Note that the dimension of the state space in this case is three, and a grid-based method becomes extremely inefficient. Hence, developing a numerical solver with good scalability is crucial. As we mentioned earlier, and as described in next section, we will use the regression Monte Carlo idea and Gaussian process surrogates to  compute the optimal pair $(\hat \varphi,\hat \psi)$ via backward recursion.

\section{Machine learning algorithm and numerical results}\label{sec:alg-and-num-res}

It is important to note that even though the market model of Section~\ref{sec:ex} is the same as the one considered in \cite[Section~4]{BCCCJ2019}, the Bellman equations associated to the  problem in Section~\ref{sec:ex}  are  more difficult to treat numerically than those from \cite[Section~4]{BCCCJ2019}. In  \cite{BCCCJ2019} the authors used a (classical) non-machine-learning based algorithm to solve numerically the Bellman equations, which can not be used in the current work, for reasons outlined below.


The essence of the machine learning algorithm that we will use solving numerically the example from previous section is the same for both Case~I and Case~II. The algorithm begins with discretizing the relevant state space, for which we employ the regression Monte Carlo method to create a random (non-gridded) mesh for the process $Y=(W,C)$. Note that  the component $W$  depends on the control process, hence at each time $t$ we randomly select from the set $A$ a value of $\varphi_t$, and we randomly generate a value of $r^S_{t+1}$, so to simulate the value of $W_{t+1}.$  The resulting  random mesh consists of a number of simulated paths of $Y$. Then, we solve the equations \eqref{Bel1-MV}--\eqref{Bel3-MV} in Case I, and their counterparts in Case II, and compute the optimal trading strategies at all mesh points.

The need for applying machine learning to solve our Bellman equations  is twofold. On one hand, to  approximate the integral operations such as $Q^{a,\theta}_{y,t}V_{t+1}$, we  replace the integrals with weighted sums through Monte Carlo simulation or a Gaussian quadrature recipe. Accordingly, interpolation and/or extrapolation, via appropriate Gaussian Processes (GP) surrogates, will be used to evaluate the terms in the summations. Note that the state space used in the adaptive robust control method is $E_Y$, which is potentially highly dimensional, where traditional linear interpolation/extrapolation methods bring multiple limitations, and therefore GP surrogates are used to overcome these limitations. On the other hand, the computation procedure involving solving the equations \eqref{Bel1-MV}--\eqref{Bel3-MV} in Case I, and their counterparts in Case II, outputs approximate values of the optimal strategies for the mesh points on the sample paths only. Hence, to obtain the value of $\hat{\varphi}_t(y)$ for arbitrary $y\in E_Y$, an efficient regression model for $\hat{\varphi}$, such as a GP surrogate, is desirable.

\subsection{Description of the algorithm}
In view of the above mentioned computational challenges, we numerically tackle the adaptive robust stochastic control problem by following the novel method introduced in \cite{ChenLudkovski2019}. The key idea of this method is to utilize a non-parametric value function approximation strategy (cf. \cite{Powell2007}) called Gaussian process surrogate (cf. \cite{Rasmussen2006}). For the purpose of solving the Bellman equations \eqref{Bel1-MV}--\eqref{Bel3-MV} in Case I, and their counterparts in Case II, we build GP regression model for the value function $V_{t+1}(\,\cdot\,)$ and the operator $\hat{g}_{t+1}$ so that we can evaluate
\begin{align}\label{eq:Bel4-MV}
Q^{a,\theta}_{y,t}V_{t+1}- \gamma Q^{a, \theta}_{y,t}\hat g^2_{t+1}+ \gamma(Q^{a,\theta}_{y,t}\hat g_{t+1})^2.
\end{align}
We also construct GP regression model for the optimal control $\hat{\varphi}$. It permits us to apply the optimal strategy to out-of-sample paths without actual optimization, which allows for a significant reduction of the computational cost.

 As the GP surrogate for the value function $V_t$ we consider a regression model $\widetilde{V}_t(y)$ such that for any $y^1,\ \ldots,\ y^N\in E_Y$, with $y^i\ne y^j$ for $i\ne j$, the random variables $\widetilde{V}_t(y^1),\ \ldots,\ \widetilde{V}_t(y^N)$ are jointly normally distributed. Then, given training data $(y^i,V_t(y^i))$, $i=1,\ \ldots,\ N$, for any $y\in E_Y$, the predicted value $\widetilde{V}_t(y)$, providing an estimate (approximation) of $V_t(y)$, is given by
\begin{align*}
  \widetilde{V}(y) = \left(k(y,y^1),\ldots,k(y,y^N)\right)[\mathbf{K}+\epsilon^2\mathbf{I}]^{-1}\left(V_t(y^1),\ldots,V_t(y^N)\right)^T,
\end{align*}
where $\epsilon$ is a tuning parameter, $\mathbf{I}$ is the $N\times N$ identity matrix and the matrix $\mathbf{K}$ is defined as $\mathbf{K}_{i,j}=k(y^i,y^j)$, $i,\ j=1,\ \ldots,\ N$. The function $k(\,\cdot\,,\,\cdot\,)$ is the kernel function for the GP model, and in this work we choose the kernel as the Matern-5/2 (cf. \cite{Rasmussen2006}). Fitting the GP surrogate $\widetilde{V}_t$ means to estimate the hyper-parameters inside $k(\,\cdot\,,\,\cdot\,)$ through training data $(y^i,V_t(y^i))$, $i=1,\ \ldots,\ N$. We note that since we do not have the closed form expression for $V_t(y)$, we  numerically evaluate $V_t(y)$ instead. The GP surrogates for $\hat{g}_t$ and $\hat{\varphi}_{t}$ are obtained in an analogous way. We take $\epsilon = 10^{-5}$.

Given the mesh points $\{y_t^i,\ i=1,\ \ldots,\ N,\ t=0,\ldots,T-1\}$, the overall algorithm proceeds as follows: \\
\textit{Part A:} Time backward recursion for $t=T-1,\ldots, 0$.
\begin{enumerate}
  \item Assume that $V_{t+1}(y^i_{t+1})$, $\hat{g}_{t+1}(y^i_{t+1})$ and $\hat{\varphi}_{t+1}(y^i_{t+1})$, $i=1,\ldots,N$, are numerically approximated as $\overline V_{t+1}(y^i_{t+1})$, $\overline{\hat{g}}_{t+1}(y^i_{t+1})$ and $\overline{\hat{\varphi}}_{t+1}(y^i_{t+1})$, $i=1,\ldots,N$, respectively. Also suppose that the corresponding GP surrogates $\widetilde{V}_{t+1}$, $\widetilde{\hat g}_{t+1}$, and $\widetilde{\hat \varphi}_{t+1}$ are fitted through training data $(y^i_{t+1},\overline{V}_{t+1}(y^i_{t+1}))$, $(y^i_{t+1},\overline{\hat{g}}_{t+1}(y^i_{t+1}))$, and $(y^i_{t+1},\overline{\hat{\varphi}}_{t+1}(y^i_{t+1}))$, $i=1,\ldots,N$, respectively.
  \item For time $t$, any $a\in A$, $\theta\in\tau(t,c)$ and each $y^i_t$, $i=1,\ \ldots, N$, use one-step Monte Carlo simulation to estimate the quantities
 \begin{align*}
	Q^{a,\theta}_{y_t^i,t}V_{t+1} &= \bE_{\theta}\left[V_{t+1}(\boldsymbol G(t,y_t^i,a,Z_{t+1}))\right],\\
	Q^{a,\theta}_{y_t^i,t}\hat{g}^2_{t+1} &= \bE_{\theta}\left[\hat{g}^2_{t+1}(\boldsymbol G(t,y_t^i,a,Z_{t+1}))\right],\\
	Q^{a,\theta}_{y_t^i,t} \hat{g}_{t+1} &= \bE_{\theta}\left[\hat{g}_{t+1}(\boldsymbol G(t,y_t^i,a,Z_{t+1}))\right].
\end{align*}
For that, if $Z^1_{t+1},\ \ldots,\ Z^M_{t+1}$ is a sample of $Z_{t+1}$ drawn from the normal distribution corresponding to parameter $\theta$, where $M>0$ is a positive integer, then estimate the above expectations as
\begin{align*}
	Q^{a,\theta}_{y_t^i,t}V_{t+1}&\approx \widetilde Q^{a,\theta}_{y_t^i,t}\widetilde V_{t+1}:= \frac{1}{M}\sum_{i=1}^M\widetilde{V}_{t+1}(\boldsymbol G(t,y_t^i,a,Z^i_{t+1})),\\
	Q^{a,\theta}_{y_t^i,t}\hat{g}^2_{t+1}&\approx \widetilde Q^{a,\theta}_{y_t^i,t}\widetilde{\hat{g}}^2_{t+1} :=\frac{1}{M}\sum_{i=1}^M\widetilde{\hat g}^2_{t+1}(\boldsymbol G(t,y_t^i,a,Z^i_{t+1})),\\
	Q^{a,\theta}_{y_t^i,t}\hat{g}_{t+1}&\approx \widetilde Q^{a,\theta}_{y_t^i,t}\widetilde{\hat{g}}_{t+1}:=\frac{1}{M}\sum_{i=1}^M\widetilde{\hat g}_{t+1}(\boldsymbol G(t,y_t^i,a,Z^i_{t+1})).
\end{align*}
Next, estimate the values of \eqref{eq:Bel4-MV}.
	\item For each $y^i_t$, $i=1,\ \ldots,\ N$, and any $a\in A$, build a uniform grid over the set $\tau(t,c)$, and search for a grid point, say $\hat{\theta}(y^i_t,a)$, that minimizes
$$
\widetilde Q^{a,\theta}_{y_t^i,t}\widetilde V_{t+1}
-\gamma \widetilde Q^{a, {\theta}}_{y_t^i,t}\widetilde {\hat g}^2_{t+1}
+ \gamma (\widetilde Q^{a,\theta}_{y_t^i,t}\widetilde {\hat g}_{t+1})^2.
$$

\item Compute
$$
\overline V_t(y_t^i)=\max_{a\in A} \left\{\widetilde Q^{a,\hat{\theta}(y_t^i,a)}_{y_t^i,t}\widetilde V_{t+1}
- \gamma \widetilde Q^{a, \hat{\theta}(y_t^i,a)}_{y_t^i,t}\widetilde{\hat g}^2_{t+1}
+ \gamma(\widetilde Q^{a,\hat{\theta}(y_t^i,a)}_{y_t^i,t}\widetilde {\hat g}_{t+1})^2\right\},
$$
and obtain a maximizer $\overline{\hat{\varphi}}_t(y_t^i)$, and corresponding $\overline{\hat{g}}_t(y_t^i)=\widetilde Q^{\hat{\varphi}_t(y_t^i),\hat{\theta}(y_t^i,\hat{\varphi}_t(y_t^i))}_{y_t^i,t}\widetilde {\hat g}_{t+1}$, $i=1,\ldots,N$.
\item Fit  GP regression models for $V_t(\,\cdot\,)$ and $\hat{g}_t(\,\cdot\,)$ using the results from Step~4 above. Fit a GP model for $\hat{a}_t(\,\cdot\,)$ as well; this is needed for obtaining values of the optimal strategies for out-of-sample paths in Part B of the algorithm.
	\item Goto 1: Start the next recursion for $t-1$.
\end{enumerate}
\textit{Part B:} Forward simulation to evaluate  the performance of the GP surrogate $\widetilde{\hat \varphi}_t(\cdot)$, $t=0, \ldots,T-1$, over the out-of-sample paths.
\begin{enumerate}
  \item Draw $K>0$ samples of i.i.d.  $Z_{1}^{*,i},\ldots,Z_{T}^{*,i}$, $i=1,\ldots,K$, from the normal distribution corresponding to the assumed true parameter $\theta^*$.
  \item All paths will start from the initial state $y_0$. The state along each path $i$ is updated according to $\boldsymbol G(t,y_t^i,\widetilde{\hat \varphi}_t^i(y_t^i),Z_{t+1}^{*,i})$, where $\widetilde{\hat \varphi}_t$ is the GP surrogate fitted in Part A.
  \item Obtain the terminal wealth $\hat W^{*,i}_T$, generated by $\widetilde{\hat \varphi}$ along the path corresponding to the sample of  $Z_{1}^{*,i},\ \ldots,\ Z_{T}^{*,i}$, $i=1,\ \ldots,\ K$, and compute
           \begin{align}\label{eq:ar}
       V^{\text{ar}}:=\underbrace{\frac{1}{K}\sum_{i=1}^{K}\widehat W_T^{\text{ar},i}}_{\text{sample mean of } \widehat W^{\text{ar}}_T}-\gamma\underbrace{\left(\frac{1}{K}\sum_{i=1}^{K}\left(\widehat W^{\text{ar},i}_T\right )^2-\left(\frac{1}{K}\sum_{i=1}^{K}\widehat W^{\text{ar},i}_T\right)^2\right)}_{\text{sample variance of }\widehat W^{\text{ar}}_T}
      \end{align}
      as an estimate of the performance of the optimal adaptive robust sub-game perfect strategy $\hat \varphi$.
\end{enumerate}

We compare \eqref{eq:ar} to the performance of strategy generated by the strong robust sub-game perfect methodology (cf. Remark~\ref{remark:strong}) on $K=2000$ out-of-sample paths, where the latter performance is measured in terms of the mean-variance utility, say $V^{\text{sr}}$, which is computed in analogy to \eqref{eq:ar}.

\subsection{Numerical results}
In this section we apply the machine learning algorithm described above by taking a specific set of parameters. For both, Case~I and Case~II we take: $T=52$ with one period of time corresponding to one week; the annualized return on banking account being equal to 0.02 or equivalently $r=0.0003846$; the initial wealth $W_0=100$; in  Part A of our algorithm the number of Monte Carlo simulations is $N=200$, and $M=100$; {the number of forward simulations in Part B is taken $K=2000$}; the confidence level $\alpha=0.1$. For both cases, we analyze the performance of the control methods for $\gamma=0.2$ and $\gamma =0.9$.
The assumed true parameter values, the initial guesses for the parameters, the bounds for the uncertainty set $\boldsymbol{\Theta}$, as well as the numerical results, are presented for each case separately.

In what follows we will abbreviate adaptive robust as AR, and strong robust as SR.

\bigskip\noindent
\textit{\textbf{Case I.}} Recall that in this case only the return $\mu$ is assumed to be unknown. The assumed true parameter value is denoted by $\mu^*$, the initial guess is denoted by $\mu_0$, the uncertainty set is  the interval $\boldsymbol{\Theta}=[\underline{\mu}, \overline{\mu}]$. The relevant parameters are summarized in Table~\ref{table:mv_params}.

\begin{table}[!ht]
	\centering
	$$\begin{array}{c} 
	\hline \\
	T=52, \ r = 0.0003846, \ \gamma = 0.2, \ W_0 = 100\\ \alpha = 0.1, \ N = 200, \
	
	M=100, \ K=2000 \\ \mu^*=0.00192, \  \underline{\mu}=0.000192, \ \overline{\mu}=0.0096, \ \sigma= 0.0166
	\\
\mu_0 = 0.002308 \textrm{ \ (optimistic), or \ } \mu_0=0.001538 \textrm{\ (pessimistic)} \\ \\
	\hline
	\end{array}$$
	\caption{Model parameters for mean-variance portfolio selection problem; Case I. }
	\label{table:mv_params}
\end{table}

\begin{figure}[h!]
	\centering
	\includegraphics[width=0.7\textwidth]{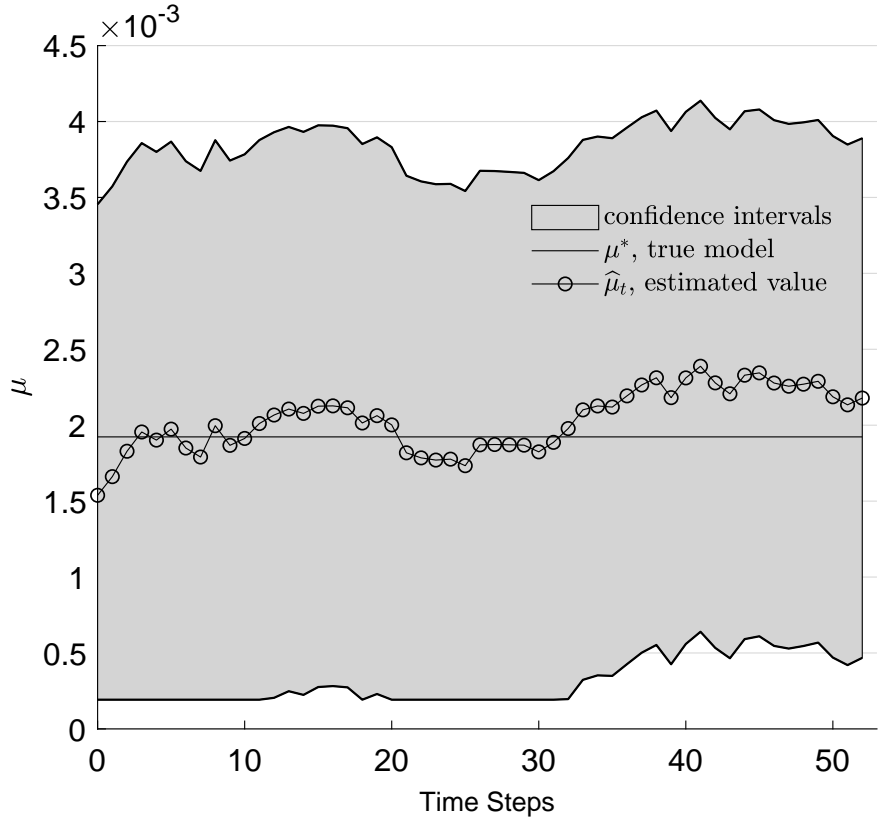}
	\caption{Evolution of the confidence intervals $\tau(t,c)$ at the confidence level $\alpha=10\%$; Case I, pessimistic.}
\label{fig:interval}
\end{figure}

We start by presenting the evolution of the confidence intervals $\tau(t,c)$ for the unknown parameter $\mu$; see Figure~\ref{fig:interval}, which represents the pessimistic initial guess (i.e. $\mu_0=  0.001538<\mu^\ast= 0.00192$). We recall that the SR methodology searches at each time for the worst-case model in $\boldsymbol{\Theta}$, while the AR searches over the confidence region $\tau(t,c)$, and then approximates  the corresponding optimal strategies.

In Figure~\ref{fig:histogramCase1} we display the histogram of out-of-sample terminal wealth $W_T$ that corresponds to the two subcases (optimistic and pessimistic) and two stochastic control approaches (AR and SR). The summary statistics are presented in Table~\ref{table:comparisonsCase1}.

\begin{figure}[h!]
   \centering

     \includegraphics[width=0.49\textwidth]{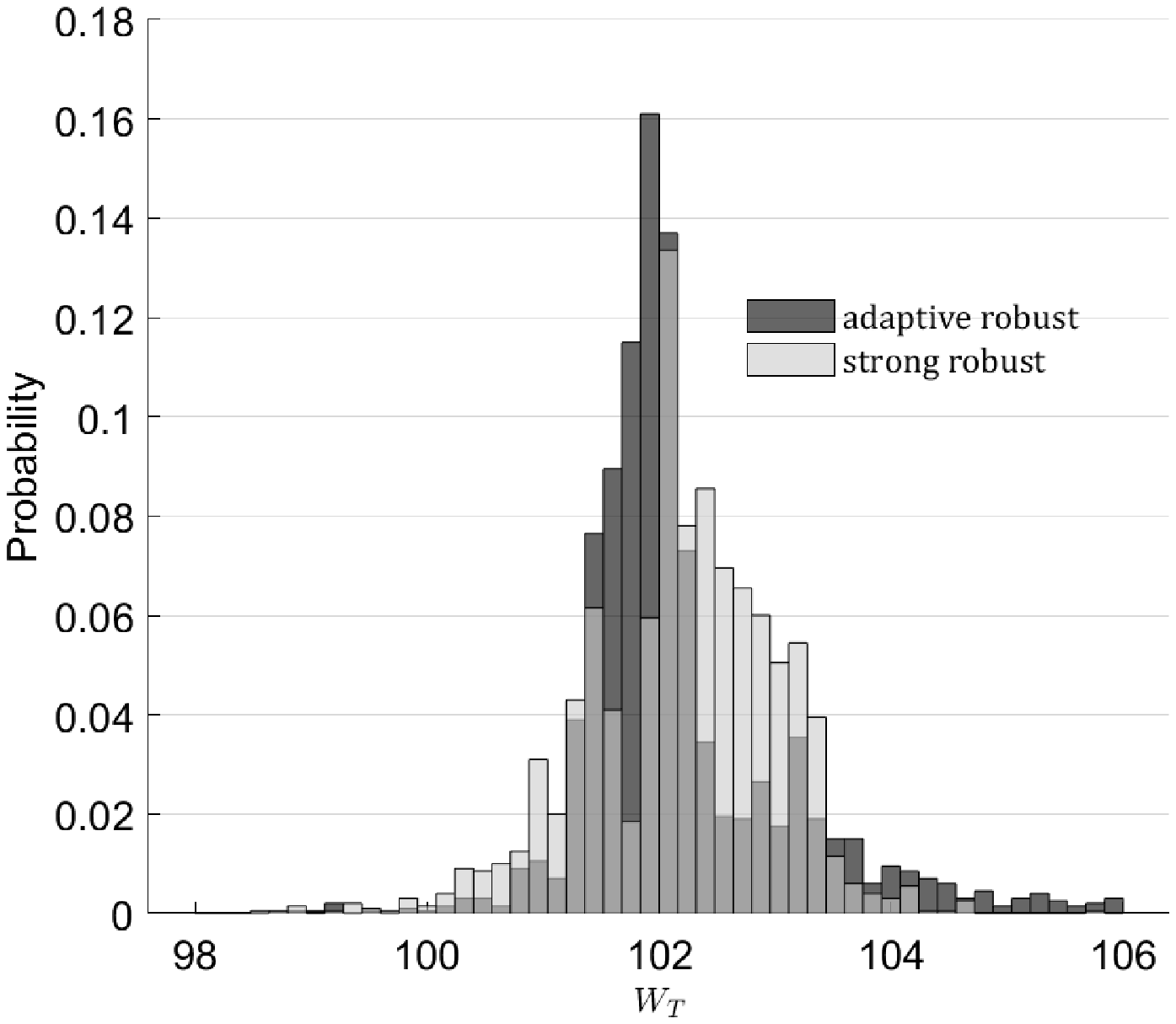}
     \includegraphics[width=0.49\textwidth]{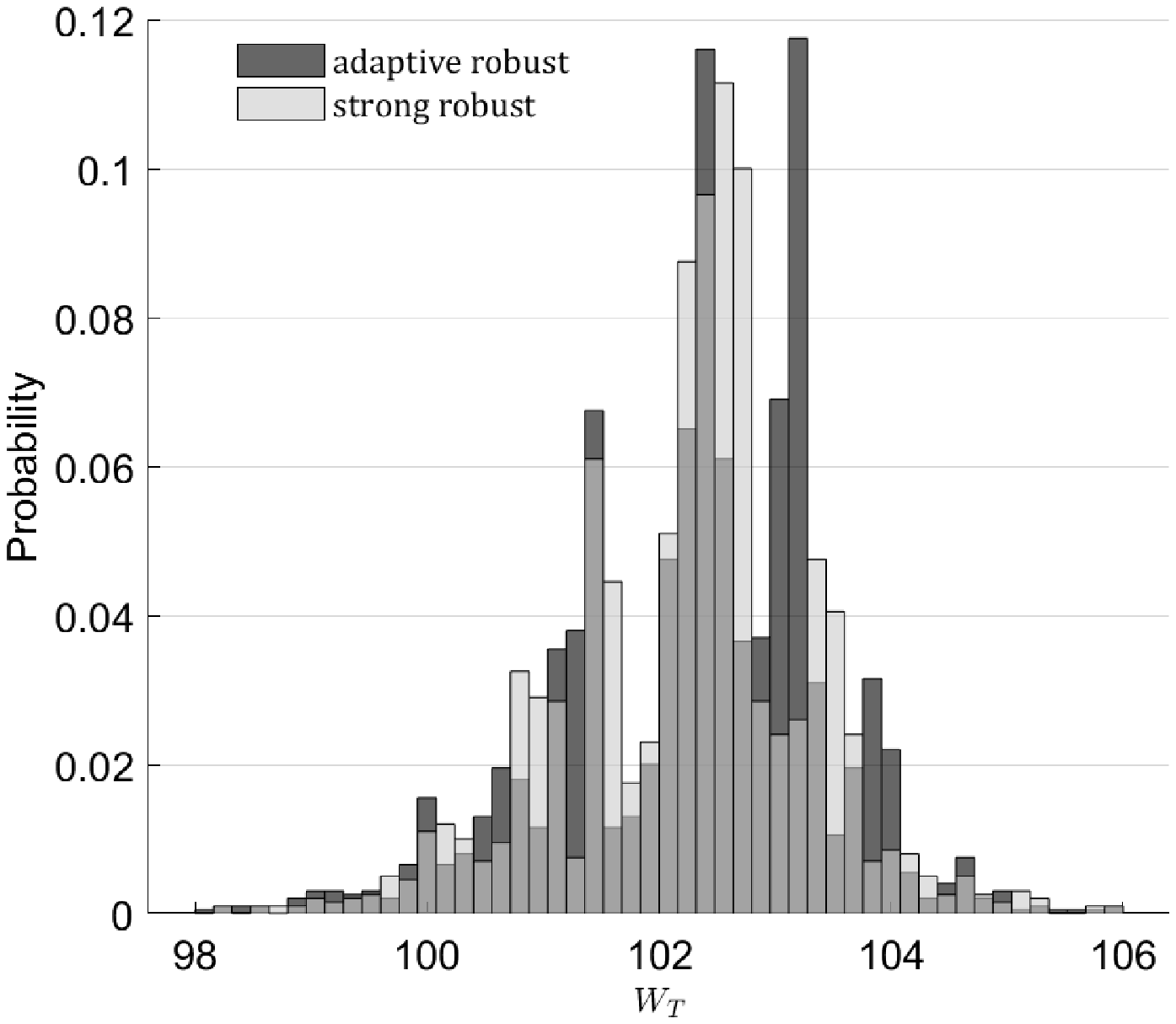}

     \includegraphics[width=0.49\textwidth]{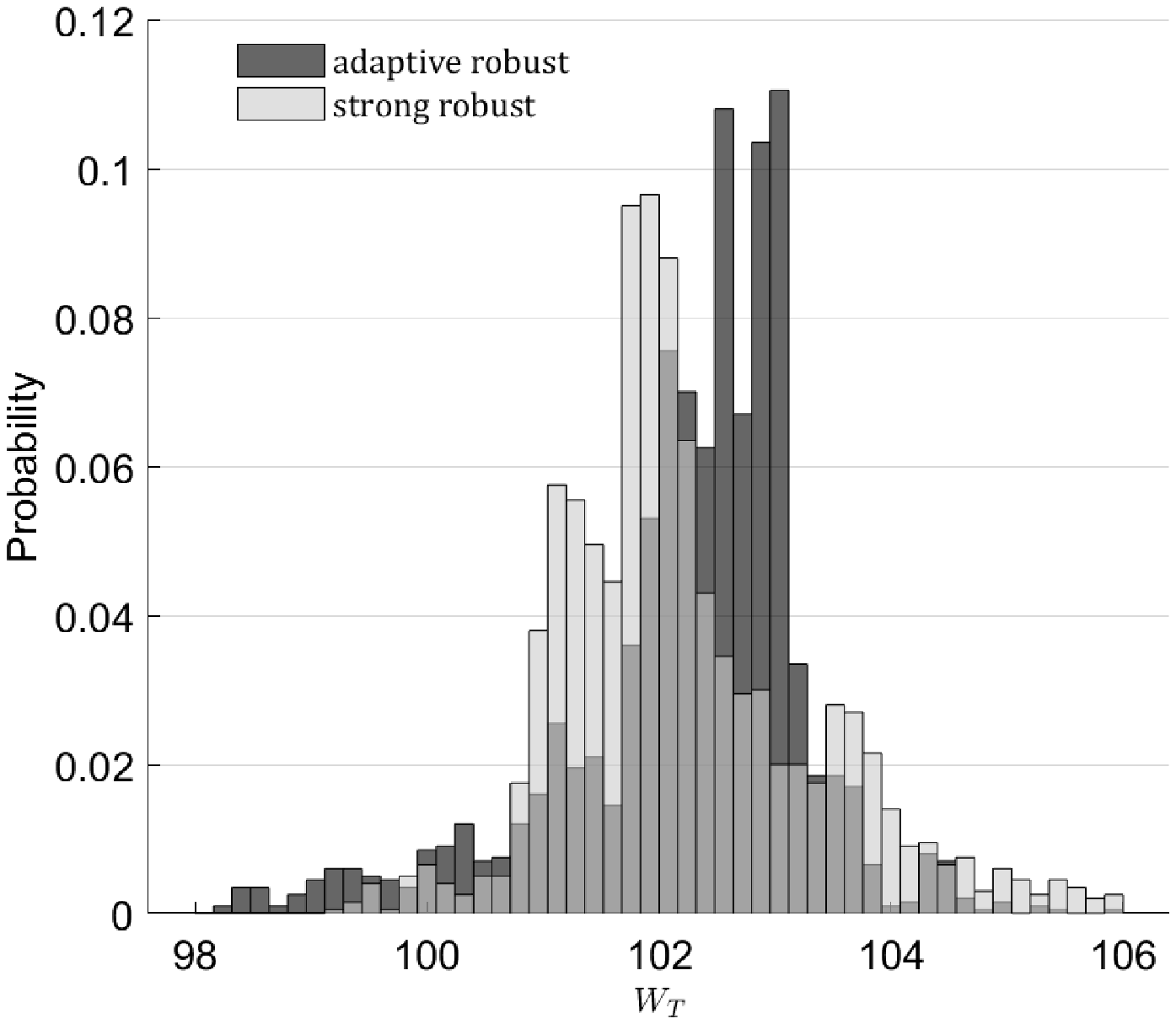}
     \includegraphics[width=0.49\textwidth]{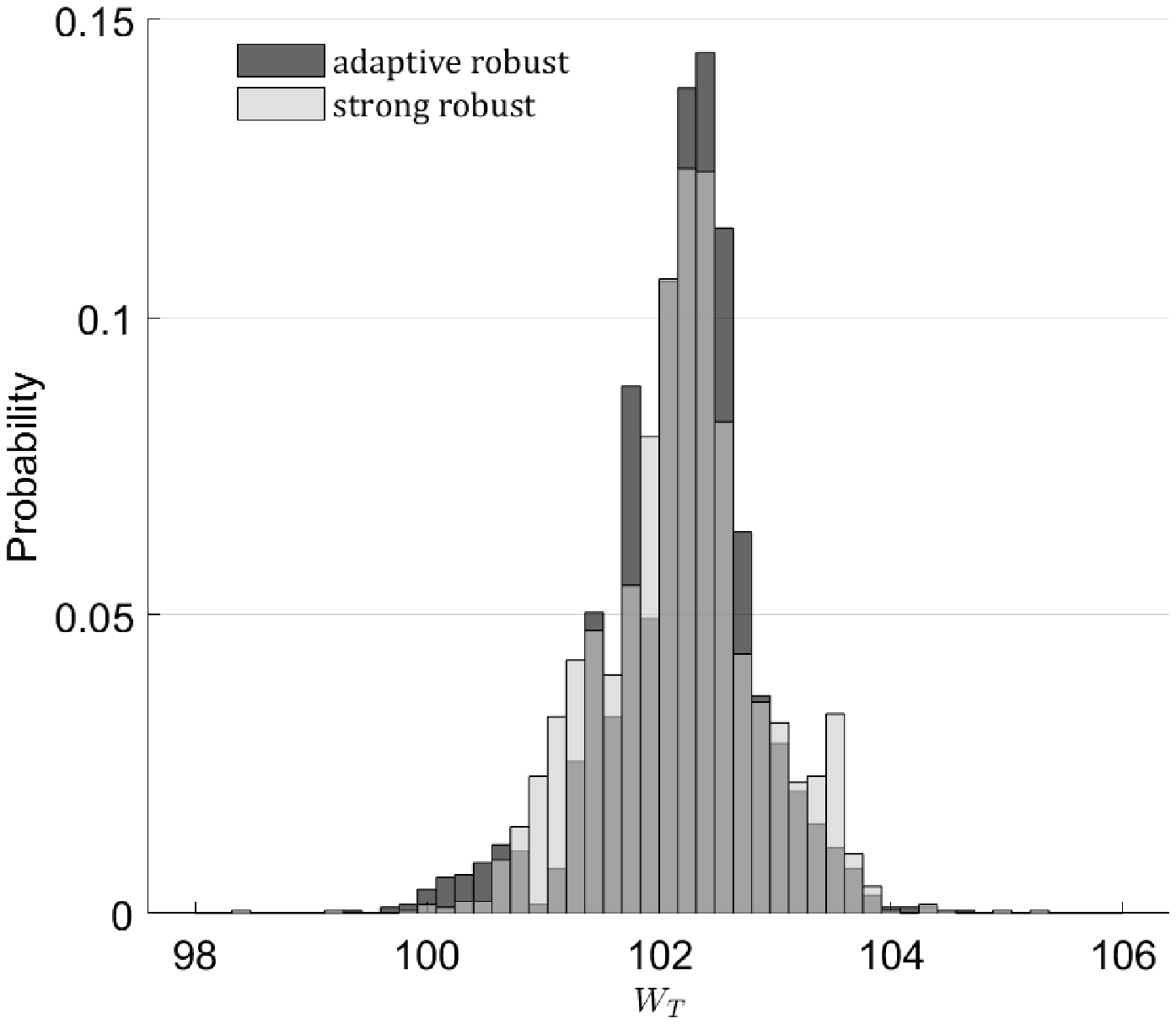}


     \caption{Histogram of the out-of-sample terminal wealth $W_T$ for Case~I, unknown $\mu$ and know $\sigma$. Risk averse coefficient $\gamma=0.2$ -- top row, and $\gamma=0.9$ -- bottom row; optimistic case -- left panel, and  pessimistic case -- right panel. }

     \label{fig:histogramCase1}
\end{figure}

\begin{table}[h!]
\centering
\renewcommand{\arraystretch}{1.3}
\begin{tabular}{c|cccc||cccc}
\hline
 \multicolumn{1}{c}{ }& \multicolumn{4}{||c||}{Optimistic}  & \multicolumn{4}{c}{Pessimistic}  \\
  \multicolumn{1}{c}{ }& \multicolumn{2}{||c}{$\gamma = 0.2$} & \multicolumn{2}{c||}{$\gamma = 0.9$} & \multicolumn{2}{c}{$\gamma = 0.2$} & \multicolumn{2}{c}{$\gamma  = 0.9$}\\
  \multicolumn{1}{c}{ }& \multicolumn{1}{||c}{AR} & SR & AR & SR & AR & SR & AR & SR \\
  \hline
   mean($W_T$) & \multicolumn{1}{||c}{102.204} & 102.203 & 102.263 & 102.267 & 102.338 & 102.262 & 102.189 & 102.190\\
 var($W_T$) & \multicolumn{1}{||c}{0.887} & 0.673 & 1.151 & 1.473 & 1.328 & 1.122 & 0.430 & 0.523 \\
 $q_{0.90}(W_T)$ & \multicolumn{1}{||c}{103.399} & 103.203 & 103.185 & 103.840 & 103.673 & 103.481 &  102.912 & 103.107 \\
 $\text{max}(W_T)$ & \multicolumn{1}{||c}{107.664} & 105.747  & 106.504 & 108.971 & 107.253 & 107.008 & 104.672 & 107.699  \\
 $\text{min}(W_T)$ & \multicolumn{1}{||c}{98.664} & 98.564 & 97.534 & 97.295 & 97.537 & 97.187 & 99.3915 & 98.352\\
 $V$ & \multicolumn{1}{||c}{102.027} & 102.068 & 101.227 & 100.941 & 102.073 & 102.038  & 101.802 & 101.719\\
 \hline
\end{tabular}
\bigskip
\caption{Mean, variance, 90\%-quantile, maximum, and minimum of the out-of-sample terminal wealth and mean-variance utility $V$ for AR and SR for Case~I.}
\label{table:comparisonsCase1}
\end{table}

We observe that the performance of the AR and SR methods is comparable in Case 1. This indicates that in this case the uncertainty reduction is not very effective. We attribute this to the fact that the uncertainty regarding the mean return requires a longer time horizon. Nevertheless, the closer inspection of the results shows that in some situations (e.g. optimistic case and $\gamma = 0.9$) the performance of AR is better than performance of SR.

\bigskip
\noindent
\textit{\textbf{Case II.}}  We take the same set of parameters as in Case~I (see Table~\ref{table:mv_params}), except that instead of the known and fixed $\sigma$, we now take
$$
\sigma^*=0.0416, \ \underline{\sigma}=0.0069, \ \overline{\sigma}=0.1109, \ \sigma_0 = 0.0347 \ (\text{optimistic}), \sigma_0 = 0.0485 \ (\text{pessimistic}).
$$
With both $\mu$ and $\sigma$ unknown, the model uncertainty set is the two dimensional rectangle $\boldsymbol{\Theta}=[\underline{\mu},\overline{\mu}]\times[\underline{\sigma}^2,\overline{\sigma}^2]$. The evolution of the projected confidence regions, which are derived  from confidence ellipsoids in this case, along with the true parameter values $(\mu^*,(\sigma^*)^2)$ and the MLEs $(\widehat\mu,\widehat{\sigma}^2)$ are displayed in Figure~\ref{fig:ellipse}.
\begin{figure}[h!]
	\centering
	\includegraphics[width=0.7\textwidth]{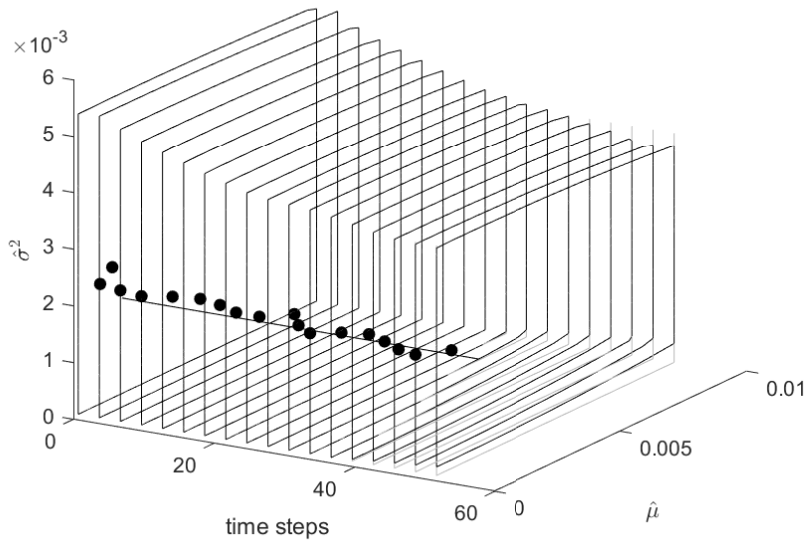}
	\caption{Evolution of $\tau(t,c)$ at confidence level $\alpha=10\%$ (ellipsoids), the true parameters value $(\mu^*,(\sigma^*)^2)$ (the solid straight line), and the MLE $(\widehat{\mu}, \widehat{\sigma}^2)$ (dotted line),  for Case II, pessimistic.  }
	\label{fig:ellipse}
\end{figure}

\begin{figure}[h!]
   \centering
     \includegraphics[width=0.49\textwidth]{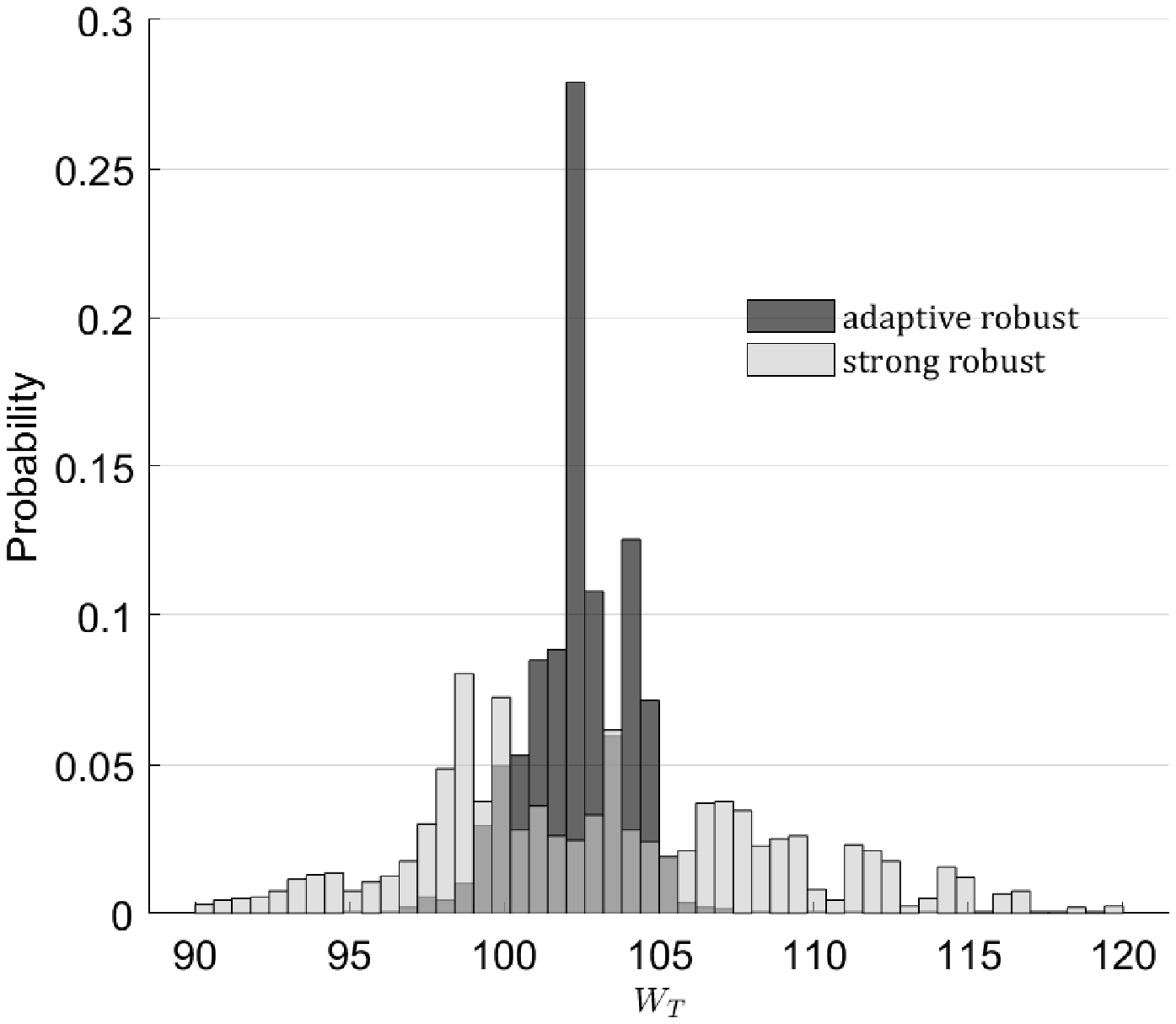}
     \includegraphics[width=0.49\textwidth]{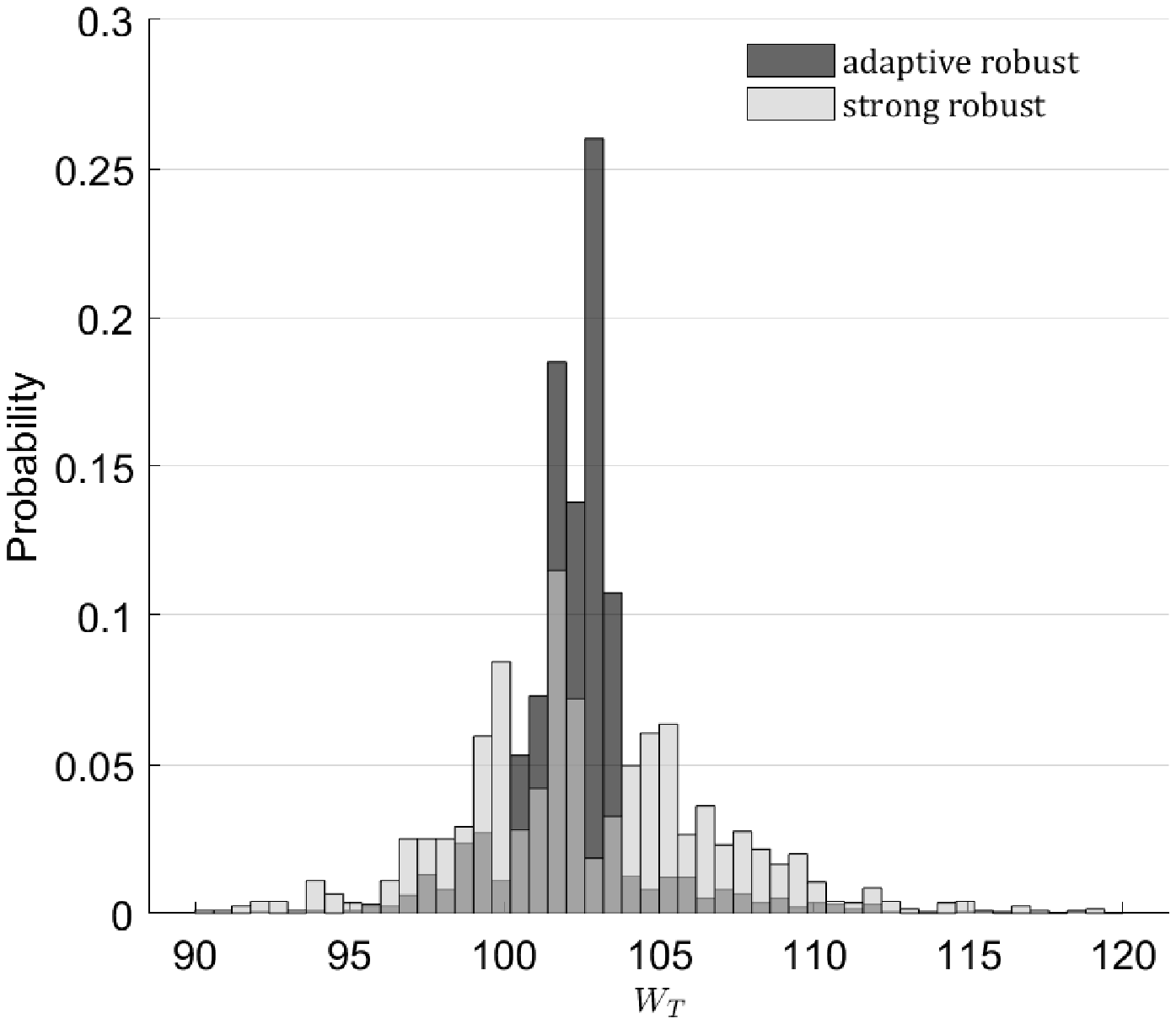}

     \includegraphics[width=0.49\textwidth]{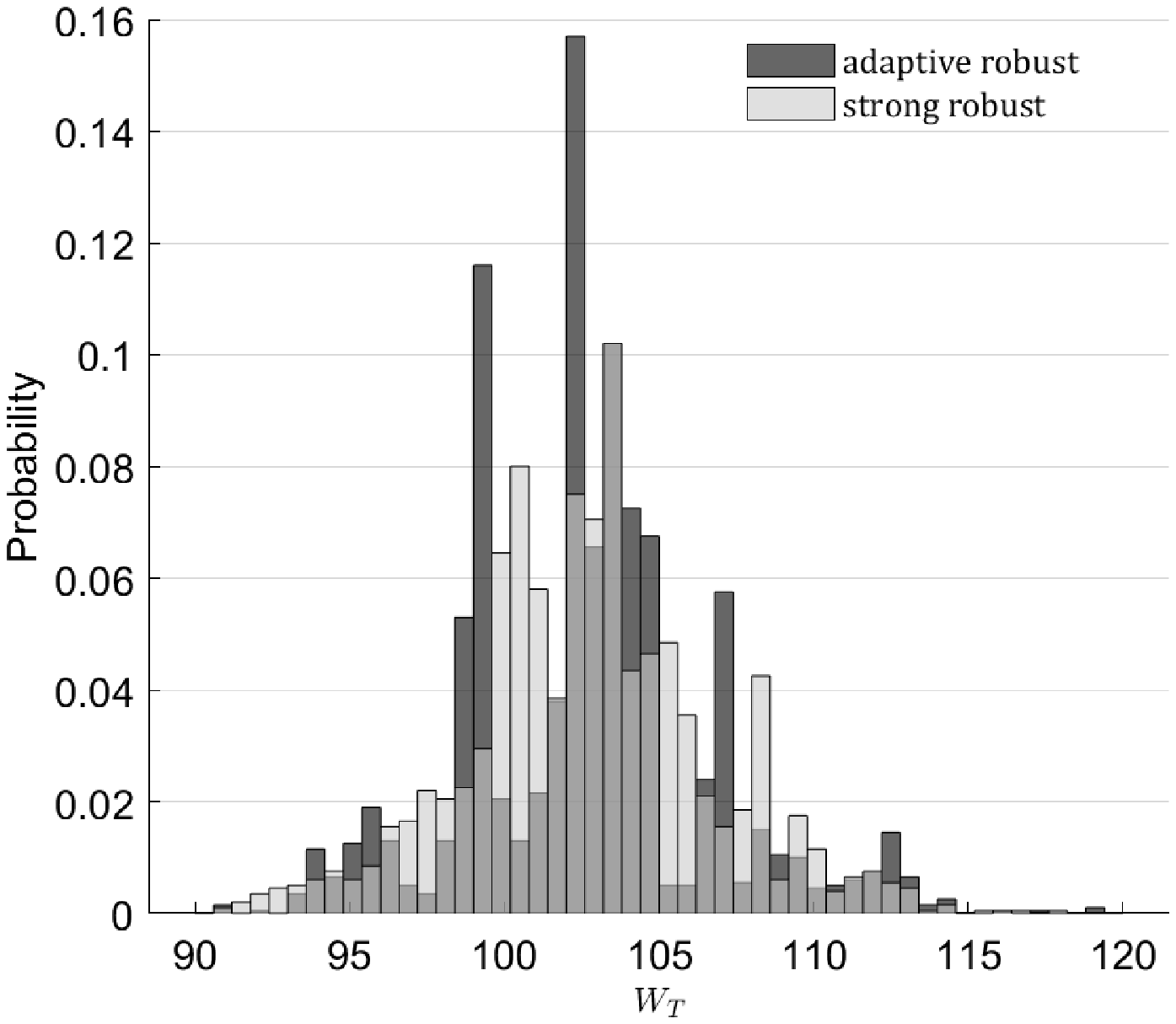}
     \includegraphics[width=0.49\textwidth]{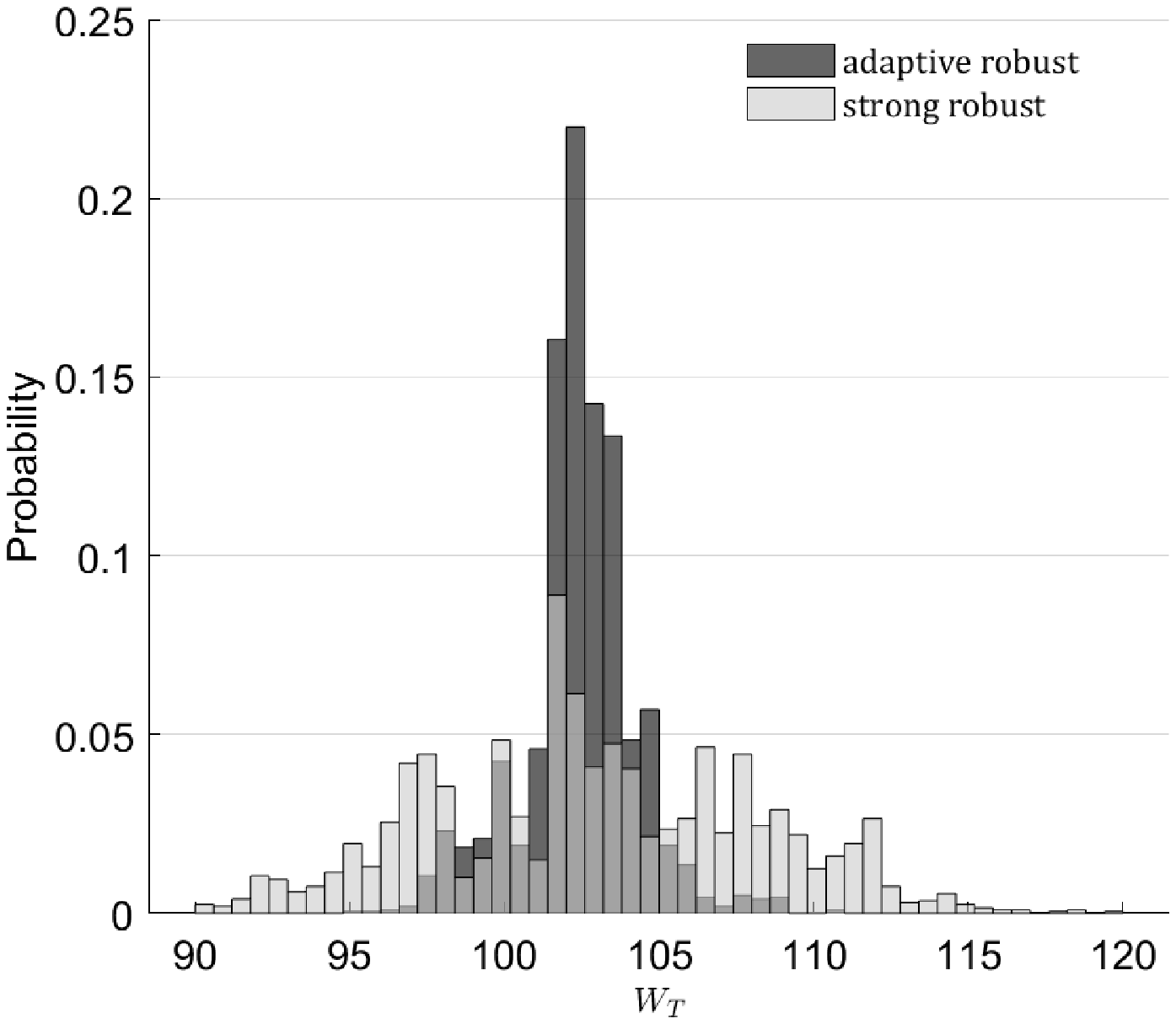}
          \caption{Histogram of the out-of-sample terminal wealth $W_T$ for Case~II, unknown $\mu$ and $\sigma$. Risk averse coefficient $\gamma=0.2$ -- top row, and $\gamma=0.9$ -- bottom row; optimistic case -- left panel, and pessimistic case -- right panel. }
     \label{fig:histogramCase2}
\end{figure}

\begin{table}[h!]
\centering
\renewcommand{\arraystretch}{1.3}
\begin{tabular}{c|cccc||cccc}
 \multicolumn{1}{c}{ }& \multicolumn{4}{||c||}{Optimistic}  & \multicolumn{4}{c}{Pessimistic}  \\
 \multicolumn{1}{c}{ }& \multicolumn{2}{||c}{$\gamma = 0.2$} & \multicolumn{2}{c||}{$\gamma = 0.9$} & \multicolumn{2}{c}{$\gamma = 0.2$} & \multicolumn{2}{c}{$\gamma = 0.9$}\\
  \multicolumn{1}{c}{ }& \multicolumn{1}{||c}{AR} & SR & AR & SR & AR & SR & AR & SR \\
  \hline
   mean($W_T$) & \multicolumn{1}{||c}{102.371} & 103.132 & 102.693 & 102.703 & 102.339 & 102.746 & 102.396 & 102.869\\
 var($W_T$) & \multicolumn{1}{||c}{2.653} & 34.654 & 15.983 & 16.149 & 4.653 & 18.911 & 3.349 & 29.698\\
 $q_{0.90}(W_T)$ & \multicolumn{1}{||c}{104.396} & 111.554 & 107.308 & 108.186 & 103.521 & 108.125 & 104.542 & 110.036\\
 $\text{max}(W_T)$ & \multicolumn{1}{||c}{113.741} & 126.470 & 123.094 & 121.139 & 118.592 & 121.682 & 110.559 & 128.299 \\
 $\text{min}(W_T)$ & \multicolumn{1}{||c}{95.027} & 81.330 & 90.838 & 87.413 & 91.981 & 80.703 & 95.238 & 82.473\\
 $V$ & \multicolumn{1}{||c}{101.840} & 96.201 & 88.309 & 88.169 & 101.408 & 98.964 & 99.382 & 76.142 \\
 \hline
\end{tabular}
\bigskip
\caption{Mean, variance, 90\%-quantile, maximum, and minimum of the out-of-sample terminal wealth and mean-variance utility $V$ for the AR and SR methods;  Case~II.}
\label{table:comparisonsCase2}
\end{table}

Similar to Case~I, we present the histograms of the out-of-sample terminal wealth $W_T$, Figure~\ref{fig:histogramCase2}, for AR and SR. The corresponding summary statistics are listed in Table~\ref{table:comparisonsCase2}.

The results clearly  indicate that overall the performance of AR is better than the performance of SR. Across the parameterizations, the value of the optimization criterion (i.e. $V_T$) is larger for AR than for SR. This is because AR produces much smaller variance of $W_T$ than SR, while both methods produce comparable values of the mean of $W_T$. This, together with the values of other statistics indicates that SR is less risky than AR. We attribute this to better handling of model uncertainty by AR than it is done by SR.

Finally, we want to mention that while we performed a similar analysis for various parameters sets and usually the obtained results are similar to the above, some cases may require a deeper analysis and understanding. For example, when the true parameter is close to the worst case one may expect that the strong robust strategy would outperform the adaptive robust strategy, which is not always the case. Also, in some examples, it may happen that the adaptive robust strategy might perform better than the strategy generated by knowing the true parameter. Understanding such phenomena will be part of the future work.

\section{Concluding remarks and future research}\label{sec:conlcusion}

In this paper we have provided a methodology for dealing with a class of time-inconsistent Markovian decision problems in discrete time subject to model uncertainty, which is also known as Knigthian uncertainty. A version of the adaptive robust approach, that was originated in \cite{BCCCJ2019}, combined with sub-game perfect approach to time-inconsistent stochastic control problems as in \cite{BjoerkMurgoci2014}, have been successfully used here.
For simplicity we were assuming that the set of available actions is finite. This assumption, although quite fine from the numerical perspective, will be generalized to an appropriate compactness assumption in a follow up study. We only proved the existence of a pair of sub-game perfect strategies in Section \ref{sec:existence}. The study of the uniqueness of such strategies is deferred to a follow-up paper. Finally, we want to mention that while in this paper we only studied time-inconsistent Markovian decision problems with terminal cost, the generalizations to the case of terminal plus running cost will be addressed in future works.

\section*{Acknowledgements}
Tomasz R. Bielecki and Igor Cialenco acknowledge support from the National Science Foundation grant DMS-1907568.

\bibliographystyle{alpha}
\bibliography{MathFinanceMaster-08-25-2020}

\end{document}

%% file: headingsIgor2020.tex
\usepackage{amsthm, amsmath, amssymb, amsfonts, graphicx, epsfig}

\usepackage{bbm}
\usepackage{mathrsfs}

\usepackage{epstopdf}

\usepackage{graphicx}
\usepackage[font=sl,labelfont=bf]{caption}
\usepackage{subcaption}

\usepackage{float}
\restylefloat{table}

\usepackage{enumerate}
\usepackage[colorlinks=true, pdfstartview=FitV, linkcolor=blue,
            citecolor=blue, urlcolor=blue]{hyperref}
\usepackage[usenames]{color}

\usepackage{url}
\makeatletter\def\url@leostyle{%
 \@ifundefined{selectfont}{\def\UrlFont{\sf}}{\def\UrlFont{\scriptsize\ttfamily}}} \makeatother

\usepackage[framemethod=default]{mdframed}

\usepackage[margin=1.0in, letterpaper]{geometry}

\newtheorem{theorem}{Theorem}

\newtheorem{lemma}[theorem]{Lemma}

\theoremstyle{definition}
\newtheorem{definition}[theorem]{Definition}
\newtheorem{example}[theorem]{Example}
\theoremstyle{remark}
\newtheorem{remark}[theorem]{Remark}

\numberwithin{equation}{section}
\numberwithin{theorem}{section}

\definecolor{Red}{rgb}{0.9,0,0.0}
\definecolor{Blue}{rgb}{0,0.0,1.0}

\def\cA{\mathcal{A}}

\def\cE{\mathcal{E}}

\def\cQ{\mathcal{Q}}

\def\cT{\mathcal{T}}

\def\bE{\mathbb{E}}
\def\bF{\mathbb{F}}

\def\bN{\mathbb{N}}

\def\bP{\mathbb{P}}
\def\bQ{\mathbb{Q}}
\def\bR{\mathbb{R}}

\def\sF{\mathscr{F}}

\newcommand{\wt}{\widetilde}

\newcommand{\set}[1]{\{#1\}}            
\newcommand{\Set}[1]{\left\{#1\right\}} 
\renewcommand{\mid}{\;|\;}              

\DeclareMathOperator{\Var}{Var}          




%% file: Time-Inconsistent-09-04-20-tom.bbl
\newcommand{\etalchar}[1]{$^{#1}$}
\begin{thebibliography}{BGPW16}

\bibitem[BC10]{BasakChabakauri2010}
S.~Basak and G.~Chabakauri.
\newblock Dynamic mean-variance asset allocation.
\newblock {\em The Review of Financial Studies}, 23:2970–3016, 2010.

\bibitem[BCC17]{BCC2017}
T.~R. Bielecki, I.~Cialenco, and T.~Chen.
\newblock Recursive construction of confidence regions.
\newblock {\em Electron. J. Statist.}, 11(2):4674--4700, 2017.

\bibitem[BCC{\etalchar{+}}19]{BCCCJ2019}
T.~R. Bielecki, I.~Cialenco, T.~Chen, A.~Cousin, and M.~Jeanblanc.
\newblock Adaptive {R}obust {H}edging {U}nder {M}odel {U}ncertainty.
\newblock {\em SIAM J. Control Optim.}, 57(2):925--946, 2019.

\bibitem[BCP17]{BCP2014}
T.~R. Bielecki, I.~Cialenco, and M.~Pitera.
\newblock A survey of time consistency of dynamic risk measures and dynamic
  performance measures in discrete time: {LM}-measure perspective.
\newblock {\em Probability, Uncertainty and Quantitative Risk}, 2(3):1--52,
  2017.

\bibitem[BCP18]{BCP2014a}
T.~R. Bielecki, I.~Cialenco, and M.~Pitera.
\newblock A unified approach to time consistency of dynamic risk measures and
  dynamic performance measures in discrete time.
\newblock {\em Mathematics of Operations Research}, 43(1):204--221, 2018.

\bibitem[BET10]{BouchardEtAl2010}
B.~Bouchard, R.~Elie, and N.~Touzi.
\newblock Stochastic target problems with controlled loss.
\newblock {\em SIAM Journal on Control and Optimization}, 48(5):3123--3150,
  2010.

\bibitem[BGPW16]{Bannister2016}
H.~Bannister, B.~Goldys, S.~Penev, and W.~Wu.
\newblock Multiperiod mean-standard-deviation time consistent portfolio
  selection.
\newblock {\em Automatica}, 73:15--26, 2016.

\bibitem[BM14]{BjoerkMurgoci2014}
T.~Bj{\"o}rk and A.~Murgoci.
\newblock A theory of {M}arkovian time-inconsistent stochastic control in
  discrete time.
\newblock {\em Finance and Stochastics}, 18(3):545--592, 2014.

\bibitem[BMZ14]{BjoerkEtAl2014}
T.~Bj\"{o}rk, A.~Murgoci, and X.~Y. Zhou.
\newblock Mean-variance portfolio optimization with state-dependent risk
  aversion.
\newblock {\em Mathematical Finance}, 24:1--24, 2014.

\bibitem[BR11]{Baeuerle2011book}
N.~B{\"a}uerle and U.~Rieder.
\newblock {\em Markov decision processes with applications to finance}.
\newblock Universitext. Springer, Heidelberg, 2011.

\bibitem[BS78]{BertsekasShreve1978Book}
D.~P. Bertsekas and S.~Shreve.
\newblock {\em Stochastic {O}ptimal {C}ontrol: {T}he {D}iscrete-{T}ime {C}ase}.
\newblock Academic Press, 1978.

\bibitem[CG91]{ChenGuo1991-Book}
H.~F. Chen and L.~Guo.
\newblock {\em Identification and stochastic adaptive control}.
\newblock Systems \& Control: Foundations \& Applications. Birkh\"auser Boston,
  Inc., Boston, MA, 1991.

\bibitem[CL19]{ChenLudkovski2019}
T.~Chen and M.~Ludkovski.
\newblock A machine learning approach to adaptive robust utility maximization
  and hedging.
\newblock {\em Preprint, arXiv:1912.00244}, 2019.

\bibitem[CLWZ12]{CuiLiWangEtAl2012}
X.~Cui, D.~Li, S.~Wang, and S.~Zhu.
\newblock Better than dynamic mean-variance: time inconsistency and free cash
  flow stream.
\newblock {\em Mathematical Finance}, 22(2):346--378, 2012.

\bibitem[DPDS01]{DuncanStettner2001}
T.~E. Duncan, B.~Pasik-Duncan, and \L~. Stettner.
\newblock Risk sensitive adaptive control of discrete time {M}arkov processes.
\newblock {\em Probab. Math. Statist.}, 21(2, Acta Univ. Wratislav. No.
  2328):493--512, 2001.

\bibitem[DPDS06]{DuncanStettner2006}
T.~E. Duncan, B.~Pasik-Duncan, and \L. Stettner.
\newblock Remarks on risk sensitive adaptive control of {M}arkov processes.
\newblock In {\em Proceedings of the 45th IEEE Conference on Decision and
  Control}, pages 2861--2865, 2006.

\bibitem[EL06]{EkelandLazrak2006}
I.~Ekeland and A.~Lazrak.
\newblock Being serious about non–commitment: subgame perfect equilibrium in
  continuous time.
\newblock {\em Preprint}, 2006.

\bibitem[EL10]{EkelandLazrak2010}
I.~Ekeland and A.~Lazrak.
\newblock The golden rule when preferences are time inconsistent.
\newblock {\em Mathematics and Financial Economics}, 4(1):29--55, 2010.

\bibitem[EP08]{EkelandPirvu2008}
I.~Ekeland and T.~Pirvu.
\newblock Investment and consumption without commitment.
\newblock {\em Mathematics and Financial Economics}, 2(1):57--86, 2008.

\bibitem[FR19]{FeinsteinRudloff2019}
Z.~Feinstein and B.~Rudloff.
\newblock Time consistency for scalar multivariate risk measures.
\newblock {\em Preprint arXiv:1810.04978}, 2019.

\bibitem[Gol80]{Goldman1980}
S.~Goldman.
\newblock Consistent plans.
\newblock {\em The Review of Economic Studies}, 47:533–537, 1980.

\bibitem[GS89]{GilboaSchmeidler1989}
I.~Gilboa and D.~Schmeidler.
\newblock Maxmin expected utility with nonunique prior.
\newblock {\em J. Math. Econom.}, 18(2):141--153, 1989.

\bibitem[He18]{He2018}
Z.~He.
\newblock {\em Equilibrium strategies for Mean-variance problem}.
\newblock PhD thesis, University of Leeds, 2018.

\bibitem[HJZ12]{HuJinZhou2012}
Y.~Hu, H.~Jin, and X.~Y. Zhou.
\newblock Time-inconsistent stochastic linear--quadratic control.
\newblock {\em SIAM Journal on Control and Optimization}, 50(3):1548--1572,
  2012.

\bibitem[HP20]{HernandezPossamai2020}
C.~Hern\'andez and D.~Possama\"i.
\newblock Me, myself and i: a general theory of non-markovian time-inconsistent
  stochastic control for sophisticated agents.
\newblock {\em Preprint arXiv:2002.12572}, 2020.

\bibitem[HS08]{HansenBookBook2008}
P.~L. Hansen and T.~J. Sargent.
\newblock {\em Robustness}.
\newblock Princeton University Press, 2008.

\bibitem[HSTW06]{HansenSargent2006}
L.~P. Hansen, T.~J. Sargent, G.~Turmuhambetova, and N.~Williams.
\newblock Robust control and model misspecification.
\newblock {\em J. Econom. Theory}, 128(1):45--90, 2006.

\bibitem[KMZ17]{KarnamEtAl2017}
C.~Karnam, J.~Ma, and J.~Zhang.
\newblock Dynamic approaches for some time-inconsistent optimization problems.
\newblock {\em Ann. Appl. Probab.}, 27(6):3435--3477, 2017.

\bibitem[KR19]{KovacovaRudloff2019}
G.~Kováčová and B.~Rudloff.
\newblock Time consistency of the mean-risk problem.
\newblock {\em Preprint arXiv:1806.10981}, 2019.

\bibitem[KV15]{KumarVaraiya2015Book}
P.~R. Kumar and P.~Varaiya.
\newblock {\em Stochastic systems: estimation, identification and adaptive
  control}, volume~75 of {\em Classics in applied mathematics}.
\newblock SIAM, 2015.

\bibitem[LN00]{LiNg2000}
D.~Li and W.-L. Ng.
\newblock Optimal dynamic portfolio selection: Multiperiod mean-variance
  formulation.
\newblock {\em Mathematical Finance}, 10(3):387--406, 2000.

\bibitem[LZ00]{LiZhou2000}
D.~Li and X.~Y. Zhou.
\newblock Continuous-time mean-variance portfolio selection: A stochastic {LQ}
  framework.
\newblock {\em Applied Mathematics and Optimization}, 42:19--33, 2000.

\bibitem[Pow07]{Powell2007}
W.~Powell.
\newblock {\em Approximate Dynamic Programming: Solving the curses of
  dimensionality}.
\newblock Wiley-Blackwell, 2007.

\bibitem[Ras06]{Rasmussen2006}
C.~E. Rasmussen.
\newblock {\em Gaussian Processes for Machine Learning}.
\newblock The MIT Press, 2006.

\bibitem[SC17]{ShiCui2017}
Y.~Shi and X.~Cui.
\newblock {\em Time Inconsistency and Self-Control Optimization Problems:
  Progress and Challenges}, pages 33--42.
\newblock Springer International Publishing, 2017.

\bibitem[Str55]{Strotz1955}
R.~H. Strotz.
\newblock Myopia and inconsistency in dynamic utility maximization.
\newblock {\em The Review of Economic Studies}, 3:165--180, 1955.

\end{thebibliography}
